\documentclass[10pt,a4paper]{article}

\usepackage[pdftex]{graphicx}
\usepackage{amsmath}
\usepackage{amssymb}
\usepackage{amsthm}
\usepackage{amsfonts}
\usepackage{wrapfig}
\usepackage{calc}
\usepackage{tabulary}
\usepackage[plain]{fullpage}
\usepackage{xcolor}
\usepackage{tikz}
\usetikzlibrary{snakes}

\newcommand{\aut}{\textnormal{Aut}}

\newcommand{\disj}{\textnormal{disj}}
\newcommand{\supp}{\textnormal{supp}}

\newcommand{\path}[1]{\textnormal{Path}\langle {#1} \rangle}

\newcommand{\cfpo}[1] {\mathrm{CFPO}_{#1}}
\newcommand{\alt}[1] {\mathrm{Alt}_{#1}}

\newcommand{\comm}{\mathrm{Comm}}
\newcommand{\conj}{\mathrm{Comm}}
\newcommand{\indec}{\mathrm{Indec}}

\makeatletter
\providecommand*{\cupdot}{%
  \mathbin{%
    \mathpalette\@cupdot{}%
  }%
}
\newcommand*{\@cupdot}[2]{%
  \ooalign{%
    $\m@th#1\cup$\cr
    \sbox0{$#1\cup$}%
    \dimen@=\ht0 %
    \sbox0{$\m@th#1\cdot$}%
    \advance\dimen@ by -\ht0 %
    \dimen@=.5\dimen@
    \hidewidth\raise\dimen@\box0\hidewidth
  }%
}

\providecommand*{\bigcupdot}{%
  \mathop{%
    \vphantom{\bigcup}%
    \mathpalette\@bigcupdot{}%
  }%
}
\newcommand*{\@bigcupdot}[2]{%
  \ooalign{%
    $\m@th#1\bigcup$\cr
    \sbox0{$#1\bigcup$}%
    \dimen@=\ht0 %
    \advance\dimen@ by -\dp0 %
    \sbox0{\scalebox{2}{$\m@th#1\cdot$}}%
    \advance\dimen@ by -\ht0 %
    \dimen@=.5\dimen@
    \hidewidth\raise\dimen@\box0\hidewidth
  }%
}
\makeatother

\newtheorem{lemma}{Lemma}[section]
\newtheorem{theorem}[lemma]{Theorem}
\newtheorem{cor}[lemma]{Corollary}

\newtheorem{dfn}[lemma]{Definition}
\newtheorem{prop}[lemma]{Proposition}

\title{The Reconstruction of Cycle-free Partial Orders from their Abstract Automorphism Groups II : Cone Transitive CFPOs}

\author{Robert Barham \\ Institut f\"ur Algebra, TU Dresden \\ robert.barham@yahoo.co.uk}

\begin{document}

\maketitle

\thanks{The author has received funding from the European Research Council under the European Community's Seventh Framework Programme (FP7/2007-2013 Grant Agreement no. 257039).}

\abstract{In this triple of papers, we examine when two cycle-free partial orders can share an abstract automorphism group.  This question was posed by M. Rubin in his memoir concerning the reconstruction of trees.

In this middle paper, we adapt a method used by Shelah in \cite{ShelahPermutation} and \cite{ShelahPermutationErrata}, and by Shelah and Truss in \cite{ShelahTrussQuotients} to define a cone transitive CFPO inside its automorphism group using the language of group theory.}

\section{Introduction}\label{intro2}

This paper draws on the methods employed in \cite{ShelahTrussQuotients}, which is about reconstructing the quotients of symmetric groups as permutations groups from the quotients of symmetric groups as abstract groups.  This paper uses $A_5$, the alternating group on five elements, chosen because it's the smallest non-abelian simple group, to represent the set being permuted.  This paper also uses $A_5$ to represent the CFPO.

We take the abstract automorphism group of a cone transitive CFPO and define the original CFPO.  Section \ref{conetrans} is devoted to properly defining the CFPOs where we apply this method.  Section \ref{section:reconstructionofconetrans} produces a long chain of first order formulae, starting with the 60-ary formula that states `these automorphisms form a subgroup isomorphic to $A_5$, the alternating group on five elements'.  There then follows a series of formulas with the goals of: defining subgroups whose support is exactly some of the extended cones of a single point; and expressing when two of these subgroups have disjoint support.  These two goals are, by far, the hardest part of this paper.  Afterwards, we have the relatively simple task of representing the points of the CFPO with these subgroups, and recovering the betweenness relation.

The final section examines how we can recover the order from the betweenness relation.  In some circumstances, the order relation is first-order definable from the betweenness relation, but not always, and certainly not with the same formula in all circumstances.  To over come this, we end this paper by giving an $L_{\omega_1, \omega}$-formula that always defines the order.

\section{Cone Transitive}\label{conetrans}

\begin{dfn}
Let $M$ be a CFPO and let $x, y, z \in M$  Such that $z < x < y$.  The \textbf{upwards cone} of $x$ that contains $y$ is the set $\lbrace t \in M \: : \: x \not\in \path{t,y} \rbrace$.  The \textbf{downwards cone} of $x$ that contains $z$ is the set $\lbrace t \in M \: : \: x \not\in \path{t,z} \rbrace$.
\end{dfn}

\begin{dfn}
We say that CFPO $M$ is \textbf{1-transitive} if for all $x,y \in M$ there exists a $\varphi \in \aut(M)$ such that $\varphi(x)=y$.
\end{dfn}

\begin{dfn}
We say that a CFPO $M$ is \textbf{cone transitive} if it is 1-transitive and if $C$ and $D$ are both upwards cones or both downwards cones of $x \in M$ then there exists a $\varphi \in \aut(M)$ such that $\varphi(C) = D$.
\end{dfn}

Since every element may be sent to any other, if $M$ is 1-transitive then $M$ is monochromatic.  $M$ is cone transitive implies that $M$ is one-transitive, so all cone transitive CFPOs are monochromatic.

Cone transitive CFPOs are the arena for an interpretation inspired by Shelah and Truss' work.  Unfortunately, there are very few Rubin complete cone transitive CFPOs.  However, these methods still work when we do not have Rubin completeness, so for this chapter we drop the assumption that $M$ is Rubin complete.

We require one additional assumption before we begin our interpretation.

\begin{dfn}
The \textbf{upwards ramification order} of $x$ in $M$, written as $Ro\uparrow (x)$, is the number of upwards cones of $x$.

The \textbf{downwards ramification order} of $x$ in $M$, written as $Ro\downarrow (x)$, is the number of downwards cones of $x$.
\end{dfn}

\begin{prop}
If $M$ is 1-transitive then for all $x$ and $y$
$$Ro \uparrow (x) = Ro \uparrow(y) \;\textnormal{ and }\; Ro \downarrow (x) = Ro \downarrow(y)$$
\end{prop}
\begin{proof}
Any automorphism that maps $x$ to $y$ also maps the cones above $x$ to the cones above $y$.  The same is true for the cones below.
\end{proof}

\begin{dfn}
Let $M$ be 1-transitive.  The \textbf{upwards} (resp. \textbf{downwards}) \textbf{ramification order} of $M$, written as $ro\uparrow (M)$ (resp. $ro \downarrow (M)$), is equal to $Ro \uparrow (x)$ (resp. $Ro \downarrow (x)$) for some $x$.
\end{dfn}

To get a sufficiently rich automorphism group we must also assume that both $ro \uparrow (M)$ and $ro \downarrow (M)$ are at least 5.

\begin{dfn}
Let $K_{Cone}$ be the class of cone transitive CFPOs such that
$$5 \leq ro \uparrow (M) \leq ro \downarrow (M)$$
for all $M \in K_{Cone}$.
\end{dfn}

\section{Reconstructing Betweenness in $K_{Cone}$}\label{section:reconstructionofconetrans}

All the CFPOs we are handling are from $K_{Cone}$, so are path complete, cone transitive, both $ro \uparrow (M)$ and $ro \downarrow (M)$ are greater than 4, and $ro \uparrow (M) \leq ro \downarrow (M)$.

We are now ready to give the interpretation of $M$ inside $\aut(M)$.  The interpretation uses pairs of subgroups isomorphic to $A_5$, the alternating group on five elements, to represent the points of the CFPO.  $A_5$ is chosen because it is the smallest non-abelian finite simple group.

\begin{dfn}\label{masterdfn}
Let $\bar{f}$, $\bar{f}_0$, $\bar{f}_1$, $\bar{g}$, $\bar{g}_0$ and $\bar{g}_1$ be 60-tuples from $\aut(M)$.
\begin{enumerate}
\item For all $\phi \in \aut(M)$, if $\phi$ preserves $X$ set-wise then $\phi|^X$, the \textbf{restriction} of $\phi$ to $X$, is the map obtained by taking the union of the standard restriction, which is a partial automorphism, and the restriction of the identity to $M \setminus X$.  Symbolically $\phi|^X := \phi|_X \cup id|_{M\setminus X}$.
This is only a total automorphism in certain circumstances which crop up often in this chapter.
\item $\bar{f}(x) := \lbrace y \in M \: : \: \exists f \in \bar{f} \: f(x)=y \rbrace$
\item $A_{5}(\bar{f})$ is the formula that states ``$\bar{f}$ satisfies the elementary diagram of $A_5$''.  This is the conjunction of formulas of the form $f_i f_j =f_k$ and $f_i f_j \not= f_k$.
\item $\conj(\bar{f},\bar{g})$ is the formula $$\alt{5}(\bar{f}) \wedge \alt{5}(\bar{g}) \wedge \bigwedge_{\substack{f_i \in \bar{f} \\ g_j \in \bar{g}}} (f_i g_j = g_j f_i)$$
\item if $\bar{f}$ and $\bar{g}$ satisfy $\alt{5}$ and $\phi \in \aut(M)$ is any automorphism then
$$
\begin{array}{r c l c r c l}
\bar{f}*\bar{g} &:=&(f_i g_i) &\quad\quad & \bar{f}^\phi & : = &(\phi f_i \phi^{-1}) \\
\phi * \bar{f} & := & (\phi f_i) & \quad\quad & \bar{f} * \phi & := &(f_i \phi)
\end{array}
$$
\item $\indec(\bar{f})$ is the formula $$ \neg \exists \bar{g}, \bar{h} (\bar{g}* \bar{h}= \bar{f} \wedge \comm(\bar{g},\bar{h}))$$
\item $\mathrm{Disj}(\bar{f},\bar{g})$ is the formula $$ \indec(\bar{f}) \wedge \indec(\bar{g}) \wedge \conj(\bar{f},\bar{g})$$
\item $[\supp(\bar{f}) \sqsubseteq \supp(\bar{g})]$ is the formula
$$
\begin{array}{l c}
\indec(\bar{f}) \wedge \indec(\bar{g}) \wedge \neg \disj(\bar{f},\bar{g}) & \wedge \\
\neg \exists \phi [\neg \disj( \bar{f}^\phi,\bar{f}) \wedge \disj( \bar{g}^\phi , \bar{g} )] & \wedge \\
\neg \exists \phi (\bar{f}^\phi = \bar{f} \wedge \bar{g}^\phi \not= \bar{g}) & \wedge \\
\end{array}
$$
\item $[\supp(\bar{g}) \sqsubset \supp(\bar{f})]$ is the formula
$$[\supp(\bar{g}) \sqsubseteq \supp(\bar{f})] \wedge \neg [\supp(\bar{f}) \sqsubseteq \supp(\bar{g})]$$
\item $\mathrm{SamePD}(\bar{f},\bar{g})$ (Same Point and Direction) is the formula
$$\forall \bar{h} ( [\supp(\bar{h}) \sqsubset \supp(\bar{f})] \leftrightarrow [\supp(\bar{h}) \sqsubset \supp(\bar{g})])$$
\item $\mathrm{RepPoint}(\bar{f}_0,\bar{f}_1)$ is the formula
$$\disj(\bar{f}_0,\bar{f}_1) \wedge \forall \bar{g} \exists \bar{h} \left( \neg \disj(\bar{g},\bar{h}) \wedge \left(
\begin{array}{r}
\mathrm{SamePD}(\bar{f_0},\bar{h}) \\
\mathrm{SamePD}(\bar{f_1},\bar{h})
\end{array} \vee\right) \right)
$$
\item $\mathrm{EqRepPoint}(\bar{f}_0,\bar{f}_1;\bar{g}_0,\bar{g}_1)$ is the formula
$$
\begin{array}{c}
\mathrm{RepPoint}(\bar{f}_0,\bar{f}_1) \wedge \mathrm{RepPoint}(\bar{g}_0,\bar{g}_1) \wedge \\
(\mathrm{SamePD}(\bar{f}_0,\bar{g}_0) \wedge \mathrm{SamePD}(\bar{f}_1,\bar{g}_1)) \vee (\mathrm{SamePD}(\bar{f}_0,\bar{g}_1) \wedge \mathrm{SamePD}(\bar{f}_1,\bar{g}_0))
\end{array}
$$
\end{enumerate}
\end{dfn}

\subsection{The Domain of the Interpretation}

\begin{lemma}\label{Lemma:A5Behaves}
If $\aut(M) \models \alt{5}(\bar{g})$ holds then $\bar{g}$ fixes at least one point.
\end{lemma}
\begin{proof}
Every transitive action of $\alt{5}$ is isomorphic to its action on a coset space $[ \alt{5} : H ]$ for some $H \leq A_{5}$.  The subgroups of $A_{5}$ can have orders 1, 2, 3, 4, 5, 6, 10, 12 and 60 and hence the possible values for $|\bar{g}(x)|$ are 60, 30, 20, 15, 12, 10, 6, 5 and 1.

Every element of $\bar{g}$ has finite order so for all $x$ we know that $\bar{g}(x)$ is an antichain.  Pick one $x$ such that $|\bar{g}(x)| \not= 1$ (possible, as $A_5$ is not the identity), so there are $g_i$ that act non-trivially.

Let
$$S:=\bigcup\limits_{x_i,x_j \in \bar{g}(x)} \path{x_i,x_j}^-$$
Since each $\path{x_i,x_j}^-$ is finite and $\bar{g}(x)$ is finite, $S$ is also finite, and therefore must be a $\cfpo{n}$ for some $n$.  In the previous paper we showed that there was a tree $T$ such that $\aut(S) \cong_P \aut(T)$.  The root of $T$ is fixed by every automorphism of $S$, and hence by every element of $\bar{g}$.
\end{proof}

\begin{lemma}
If $\supp(\bar{f})$ and $\supp(\bar{g})$ are disjoint then $\aut(M) \models \conj(\bar{f},\bar{g})$.
\end{lemma}

\begin{dfn}\label{dfn:ECC}
Let $\bar{f}$ be such that $\aut(M) \models A_5(\bar{f})$ and $C_i$ are the connected components of $\supp(\bar{f})$.  We say that $E \subseteq M$ is an \textit{extended connected component} of $\supp(\bar{f})$ if:
\begin{enumerate}
\item $E$ contains a union of $C_i$ and at most one element from $M \setminus \supp(\bar{f})$, which we call $e$;
\item if $C \subseteq E$ then $\bar{f}(C) \subseteq E$;
\item if $e$ exists then $E$ contains at least two connected components, $C_0$ and $C_1$, and $\lbrace e \rbrace = \path{C_0,C_1}$; and
\item if $D$ satisfies conditions 1-3 and $E \cap D \not=\emptyset$ then $E \subseteq D$.
\end{enumerate}
\end{dfn}

\begin{lemma}\label{lemma:ECC}
If $X$ is an extended connected component of some $\supp(\bar{f})$ then $\path{X,M \setminus X}$ is a singleton.
\end{lemma}
\begin{proof}
Condition 2 of Definition \ref{dfn:ECC} shows that $X$ is preserved setwise by $\bar{f}$, so by Lemma 3.10 of Part 1.  There are $x$ and $y$ such that $\path{X, M \setminus X}= \path{x,y}$.  Both $x$ and $y$ are fixed by $\bar{f}$, so $x,y \in M \setminus X$.

Suppose one of the cones above $x$ intersects $X$ and one of the cones below $x$ intersects $X$.  Let $U$ be the upwards extended cone and let $D$ be the downwards extended cone.  $\bar{f}(U) \cap \bar{f}(D) = \emptyset$, as $\bar{f}$ fixes $x$, so $\bar{f}(U) \cap X$ satisfies Conditions 1-3, and does not contain $X$ giving a contradiction.

Therefore we may assume that $X$ is contained in extended cones above $x$.  Let $y_0$ and $y_1$ lie in different extended cones below $x$.  The definition of extended cone guarantees that $\path{x,y_0} \cap \path{x,y_1} = \lbrace x \rbrace$, so $\path{x, y} = \lbrace x \rbrace$.
\end{proof}

\begin{lemma}\label{RestrictionSubgroups}
Let $\bar{f}$ satisfy $\alt{5}$.  If we partition $\supp(\bar{f})$ into two collections of extended connected components, which we will call $X$ and $Y$, then $(f_i |^X)$ and $(f_i |^Y)$ satisfy $\alt{5}$.
\end{lemma}
\begin{proof}
First of all, we must show that this lemma makes sense, i.e. $\bar{f}$ preserves the extended connected components of $\supp(\bar{f})$ set-wise and therefore $f_i|^X$ and $f_i|^Y$ are automorphisms.

Since the supports of $(f_i |^X)$ and $(f_i |^Y)$ are disjoint, $\conj((f_i |^X),(f_i |^Y))$ holds.  We consider the positive statements of the formula $A_5$ that $\bar{f}$ satisfies, which are of the form $f_i f_j = f_k$.

Since $f_i = f_i |^X f_i |^Y$ for all $i$ we can deduce that $$f_i |^X f_j |^X f_i |^Y f_j |^Y = f_k |^X f_k |^Y$$ and since $(f_\alpha |^X f_\alpha |^Y )|^X=f_\alpha |^X$ we conclude that $(f_i |^X)$ and $(f_i |^Y)$ satisfy all the positive statements of $\alt{5}$.  We now consider the negative statements, those of the form $f_i f_j \not= f_k$.

Repeating the argument for the positive statements allows us to deduce $$f_i |^X f_j |^X f_i |^Y f_j |^Y \not= f_k |^X f_k |^Y$$ which only guarantees that at least one of $f_i |^X f_j |^X \not= f_k |^X$ or $f_i |^Y f_j |^Y \not= f_k |^Y$.  Without loss of generality we assume that $f_i |^Y f_j |^Y \not= f_k |^Y$.  In $A_5$ there is the positive statement $f_i f_j = f_l$ for some $f_l \not=f_k$, so if $f_i |^X f_j |^X  = f_k |^X$, then $f_k |^X=f_l |^X$.

We define the homomorphism
$$\Phi : \left\lbrace 
\begin{array}{r c l}
\bar{f} & \rightarrow & (f_i |^X) \\
f_i & \mapsto & f_i|^{X}
\end{array}
\right.$$
$\Phi^{-1}(id)$ is a normal subgroup of $\bar{f}$. We have just found distinct $f_k$ and $f_l$ such that $f_k |^X = f_l |^X$, so since $A_5$ is simple, this means that $f_i|^X=id$ for all $f_i \in \bar{f}$, contradicting the fact that $X \cap \supp(\bar{f}) \not= \emptyset$.

Therefore if $A_5(\bar{f})$ then $A_5((f_i |^X))$ and $A_5((f_i |^Y))$.
\end{proof}

\begin{lemma}\label{lemma:nocancellingorbits}
If $\bar{g} * \bar{h} = \bar{f}$ and $\comm(\bar{g},\bar{h})$ then $\supp(\bar{g}),\supp(\bar{h}) \subseteq \supp(\bar{f})$.
\end{lemma}
\begin{proof}
Suppose there is an $x$ such that $g_j(x) \not= x$ for some $j$ and $f_i(x) = x$ for all $i$.  Therefore
$$
\begin{array}{l c rcl }
\forall i &\vspace{10pt} & h_i g_i(x) & = & x \\
\forall i & & g_i(x) &=& h_i^{-1} (x)
\end{array}
$$

There are $g_j$ and $g_k$ such that $g_j g_k (x) \not= g_k g_j (x)$ as $A_5$ is non-abelian, and if we substitute $h_j^{-1}$ for $g_j$ we find that
$$h_j^{-1} g_k (x) \not= g_k h_j^{-1} (x)$$
contradicting $M \models \comm(\bar{g},\bar{h})$.
\end{proof}

\begin{lemma}\label{noflipping}
Let $X$ and $Y$ be extended connected components of $\supp(\bar{f})$ and $\supp(\bar{g})$.  If $\comm(\bar{f},\bar{g})$ and $|X \cap Y| \geq 1$ then either $X \subseteq Y$ or $Y \subseteq X$.
\end{lemma}
\begin{proof}
Let $\lbrace x \rbrace = \path{X, M \setminus X}$ and $\lbrace y \rbrace = \path{Y, M \setminus Y}$.  These are singletons by Lemma \ref{lemma:ECC}.  Suppose $X \nsubseteq Y$ and $ Y \nsubseteq X$.

First suppose that $x = y$.  This means that $\path{X,Y}= \lbrace x \rbrace$, and that $X$ and $Y$ are entirely contained in the upwards and downwards extended cones of $x$, as illustrated in Figure 4.1.

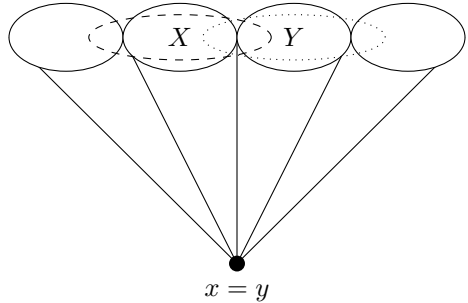
\begin{figure}[ht]\label{fig:noflipping}
\begin{center}
\begin{tikzpicture}[scale=0.15]

\draw (0,20) -- (20,0);
\draw (10,20) -- (20,0);
\draw (20,20) -- (20,0);
\draw (30,20) -- (20,0);
\draw (40,20) -- (20,0) ;

\fill[white] (5,20) ellipse (5 and 3);
\draw (5,20) ellipse (5 and 3);
\fill[white] (15,20) ellipse (5 and 3);
\draw (15,20) ellipse (5 and 3);
\fill[white] (25,20) ellipse (5 and 3);
\draw (25,20) ellipse (5 and 3);
\fill[white] (35,20) ellipse (5 and 3);
\draw (35,20) ellipse (5 and 3);

\draw[dashed] (15,20) ellipse (8 and 2);
\draw[dotted] (25,20) ellipse (8 and 2);
\draw (15,20) node {$X$};
\draw (25,20) node {$Y$};

\fill (20,0) circle (0.7);
\draw[anchor=north] (20,-1) node {$x=y$};

\end{tikzpicture}
\end{center}
\caption{If $x=y$ in Lemma \ref{noflipping}}
\end{figure}

Recall Definition \ref{dfn:ECC}, and note that $\bar{g}(X \cap Y)$ satisfies conditions 1 and 3 because both $X$ and $Y$ do, and by definition it satisfies condition 2.  Therefore $Y \subseteq \bar{g}(X \cap Y)$.  Thus if $\bar{g}(X \cap Y) \subseteq X$ then $Y \subseteq X$ and we are done.  Similarly if $\bar{f}(X \cap Y) \subseteq Y$ then $X \subseteq Y$ and we are done.

We now suppose that there is a $z \in X \cap Y$ such that $\bar{f}(z) \nsubseteq Y$ and $\bar{g}(z) \nsubseteq X$.  Let $C_z$ be the extended cone of $x$ that contains $z$.  We consider the action of $\bar{f}$ and $\bar{g}$ on the set $\bar{f}(C_z) \cup \bar{g} (C_z)$.

Let $f_i \in \bar{f}$ map $C_z$ into $X \setminus Y$ and let $g_j$ map $C_z$ into $Y \setminus X$.  Then
$$f_i g_j (C_z) = g_j(C_z)  \;\mathrm{and}\; g_j f_i (C_z) = f_i(C_z)$$
contradicting the assumption that $\aut(M) \models \comm(\bar{f},\bar{g})$.  This is depicted in Figure 20.

\begin{figure}[ht]\label{fig:C_z}
\begin{center}
\begin{tikzpicture}[scale=0.18]
\draw (20,0) ellipse (15 and 10);
\draw (40,0) ellipse (15 and 10);

\draw[->] (30,0) .. controls (40,10) and (40,-10) .. (47.9,-2.1); 
\draw[->] (30,0) .. controls (20,-10) and (20,10) .. (12.1,2.1); 

\fill[white] (30,0) circle (3);
\draw (30,0) circle (3);
\draw (30,0) node {$C_z$};

\fill[white] (10,0) circle (3);
\draw (10,0) circle (3);
\draw (10,0) node {$f_i(C_z)$};

\fill[white] (50,0) circle (3);
\draw (50,0) circle (3);
\draw (50,0) node {$g_j(C_z)$};

\draw (42,0) node {$g_j$};
\draw (18,0) node {$f_i$};

\draw (20,7) node {$X$};
\draw (40,-7) node {$Y$};
\end{tikzpicture}
\end{center}
\caption{Images of $C_z$}
\end{figure}
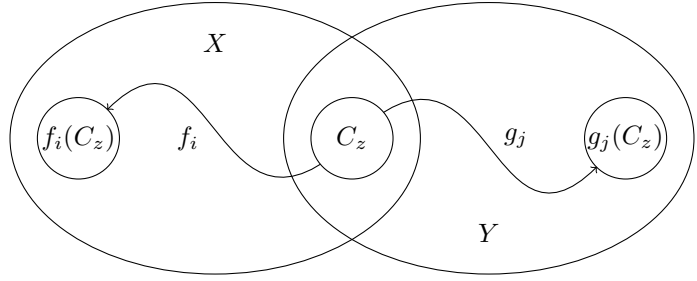

Now suppose that $x \not= y$.  Suppose $x \not\in Y$ and $y \in X$, and let $z \in Y$.  By definition, $y \in \path{z,x}$, and since $x$ is an endpoint of that path, $\path{z,x} \subseteq X$, and so $z \in X$.  This is depicted in Figure 21.

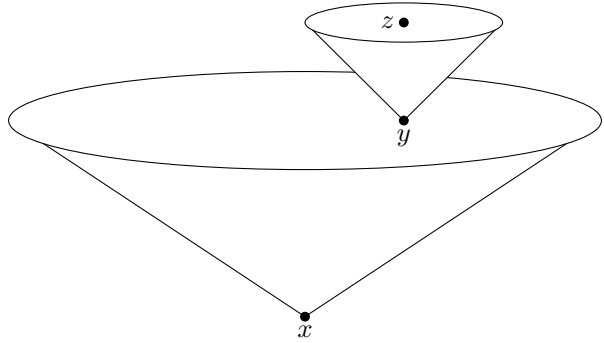
\begin{figure}[ht]\label{fig:xnotinY}
\begin{center}
\begin{tikzpicture}[scale=0.13]
\draw (0,20) -- (30,0) -- (60,20);
\fill[white] (30,20) ellipse (30 and 5);
\draw (30,20) ellipse (30 and 5);
\fill (30,0) circle (0.5);

\fill[white] (40,20) -- (30,30) -- (50,30) -- (40,20);
\draw (30,30) -- (40,20) -- (50,30);
\fill[white] (40,30) ellipse (10 and 2);
\draw (40,30) ellipse (10 and 2);
\fill (40,20) circle (0.5);

\fill (40,30) circle (0.5);

\draw[anchor=north] (30,0) node {$x$};
\draw[anchor=north] (40,20) node {$y$};
\draw[anchor=east] (40,30) node {$z$};

\end{tikzpicture}
\end{center}
\caption{$x \not\in Y$ and $y \in X$}
\end{figure}

If both $x \not\in Y$ and $y \not\in X$ then $\path{x,y} \subseteq M \setminus (X \cup Y)$.  This is depicted in Figure 22.  Let $z \in Y$.  By definition $y \in \path{x,z}$ and since $\path{x,y} \nsubseteq X$, we know that $z \not\in X$.  Similarly, if $z \in X$ then $z \not\in Y$, contradicting the assumption that $X \cap Y \not= \emptyset$.

\begin{figure}[ht]\label{fig:noflippingfinal}
\begin{center}
\begin{tikzpicture}[scale=0.2]
\draw (0,10) -- (15,0) -- (30,10);
\draw[dashed] (15,10) -- (30,0) -- (45,10);

\fill[white] (30,10) ellipse (15 and 3);
fill[white] (15,10) ellipse (15 and 3);
\draw (15,10) ellipse (15 and 3);
\draw[dashed] (30,10) ellipse (15 and 3);
\draw[dotted] (15,0) -- (22.5,8.5) -- (30,0);

\fill (15,0) circle (0.5);
\fill (30,0) circle (0.5);
\draw[anchor=north] (15,0) node {$x$};
\draw[anchor=north] (30,0) node {$y$};

\draw[anchor=south] (22.5,8.5) node {$\path{x,y}$};

\draw (10,10) node {$X$};
\draw (35,10) node {$Y$};
\end{tikzpicture}
\end{center}
\caption{$x \not\in Y$ and $y \not\in X$}
\end{figure}
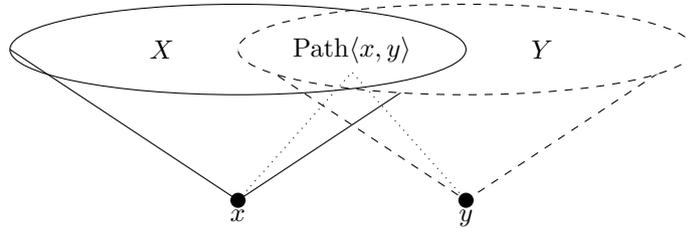

We therefore suppose that $x \in Y$ and $y \in X$.

$x \in \path{y,f_i(y)}$ for any $f_i$, as otherwise $X$ will not be an extended connected component.

Path-betweenness is preserved by automorphisms, so $g_j(x) \in \path{g_j(y),g_j f_i(y)}$ and $\bar{f}$ and $\bar{g}$ commute, and $y$ is fixed by $\bar{g}$, hence $g_j(x) \in \path{y,f_i(y)}$.  By symmetry $y \in \path{x,g_j(x)}$ and $f_i(y) \in \path{x,g_j(x)}$.  From these facts we can deduce the path-configuration of $x$, $y$, $g_j(x)$ and $f_i(y)$.
\begin{figure}[h]
\begin{center}
\begin{tikzpicture}[scale=0.3]
\draw (5,5) -- (15,5);
\fill (5,5) circle (0.07);
\fill (15,5) circle (0.07);

\draw (5,4.4) node {$x$};
\draw (15,4.4) node {$y$};
\draw (10,5.8) node {$\path{x,y}$};
\end{tikzpicture}
\end{center}
\caption{$\path{x,y}$}
\end{figure}

Since $y \in \path{x,g_j(x)}$ and $x \in \path{y,f_i(y)}$ we may add to Figure 23 $f_i(y)$ and $g_j(x)$ to obtain Figure 24.
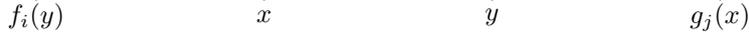
\begin{figure}[h]
\begin{center}
\begin{tikzpicture}[scale=0.3]
\draw (-5,5) -- (25,5);
\fill (-5,5) circle (0.08);
\fill (5,5) circle (0.08);
\fill (15,5) circle (0.08);
\fill (25,5) circle (0.08);

\draw (5,4.2) node {$x$};
\draw (15,4.2) node {$y$};
\draw (25,4.2) node {$g_j(x)$};
\draw (-5,4.2) node {$f_i(y)$};
\end{tikzpicture}
\end{center}
\caption{$\path{x,y}$, $f_i(y)$ and $g_j(x)$}
\end{figure}

But we also know that $f_i(y) \in \path{x,g_j(x)}$, so we deduce that $f_i(y)=x$.  Similarly $g_j(x) \in \path{y,f_i(y)}$ shows that $g_j(x) = y$.  This contradicts the fact that $\bar{f}$ fixes $x$ and $\bar{g}$ fixes $y$, so we conclude that either $X \subseteq Y$ or $Y \subseteq X$.

\end{proof}

\begin{lemma}\label{longorbits}
If $\aut(M) \models \comm(\bar{f},\bar{g})$ and $\supp(\bar{f}) \cap \supp(\bar{g}) \not= \emptyset $ then $\bar{f} * \bar{g}$ has an orbit of length 20 in $\supp(\bar{f}) \cap \supp(\bar{g})$.  If $\bar{f} * \bar{g}$ has an orbit of length 20 then it also has a non-trivial orbit of some length other than 20.
\end{lemma}
\begin{proof}
Lemma 3.5 of \cite{ShelahTrussQuotients} is:

``Suppose that $\bar{f}, \bar{g}$ are subgroups of $\mathrm{Sym}(\mathcal{X})$ isomorphic to $A_5$ (in the specified listings) which centralize each other, and such that $\langle \bar{f}, \bar{g} \rangle$ is transitive on $\mathcal{X}$.  Then $\bar{f} * \bar{g}$ has an orbit of length 20.  Moreover, if $\bar{f} * \bar{g}$ has an orbit of length 20 then is also has an orbit of some other length greater than 1.''

Let $\lbrace A_i \: : \: i \in I \rbrace$ be the ECC of $\supp(\bar{f})$ and let $\lbrace B_j \: : \: j \in J \rbrace$ be the ECC of $\supp(\bar{g})$.  Lemma \ref{noflipping} shows that if $A_i \cap B_j \not= \emptyset$ then $A_i \subseteq B_j$ or $B_j \subseteq A_i$.

Pick one such $A$ and $B$, and without loss of generality assume that $A \subseteq B$.  Let $X$ be a connected component of $A$.
$$\mathcal{X} := \langle \bar{f}, \bar{g} \rangle (X) $$
Each member of $\mathcal{X}$ is a translate of $X$.

We define $\phi_{f} :\bar{f} \rightarrow \mathrm{Sym}(\mathcal{X})$ as follows:
$\phi_{f}(f_i) = (X \mapsto f_i(X))$.
This is a homomorphism, and since $A_5$ is simple, so $\phi$'s kernel is trivial, and $\phi_f(\bar{f}) \cong A_5$.  Similarly, if we define $\phi_{g} :\bar{g} \rightarrow \mathrm{Sym}(\mathcal{X})$ as follows:
$\phi_{g}(g_i) = (X \mapsto g_i(X))$.
then $\phi_g(\bar{g}) \cong A_5$.

The `specified listings' in Shelah and Truss' Lemma 3.5 refers to the fact that the formula $A_5(\bar{f})$ will be different depending on how we enumerate $A_5$.  For example, we could insist that $f_0$ is the identity, and this would give a different formula to if we insisted that $f_5$ is the identity.  Our formula $A_5$ is fixed so we need not worry about this assumption.

$\langle \phi_f(\bar{f}), \phi_g(\bar{g}) \rangle$ is transitive on $\mathcal{X}$ since $\mathcal{X}$ is an orbit of $\langle \bar{f}, \bar{g} \rangle$.

Therefore Lemma 3.5 of \cite{ShelahTrussQuotients} is applicable to $\mathcal{X}$.
\end{proof}

\begin{lemma}\label{no60}
If $\aut(M) \models A_5(\bar{f})$ then no orbit of $\bar{f}$ has length 60.
\end{lemma}
\begin{proof}
Let $x \in M$ be such that $|\bar{f}(x)| = 60$, and let $\lbrace x_0, \ldots x_{59} \rbrace$ be an enumeration of $\bar{f}(x)$. Take $X(5,60,\mathbb{Z})$, and pick an arbitrary $z \in X(5,60,\mathbb{Z})$, and label the successors of $z$ as $z_0, \ldots z_{59}$.  For each $f_i \in \bar{f}$, let $g_i \in \aut(X(5,60,\mathbb{Z})$ be induced by the partial automorphism
$$z_n \mapsto z_m \: \mathrm{if} \; f_i(x_n)=x_m$$
$\aut(X(5,60,\mathbb{Z})) \models A_5(\bar{g})$ and $\bar{g}$ has an orbit of length 60.  Let $C_i$ be the extended cone of $z$ that contains $z_i$.

For each $y \in \bar{g}(z)$ there is a unique $g_i \in \bar{g}$ such that $g_i(z)=y$, so we may label $\bar{g}(x)$ by elements of $\bar{g}$.  In this way, we can view the action of $\bar{g}$ on $\bar{g}(x)$ as left multiplication.

We define $\bar{h}$ on each $g \in \bar{h}$ as follows:
$$h_i : g \mapsto g g_i^{-1}$$
This $\bar{h}$ commutes with $\bar{g}$, as
$$
\begin{array}{rcl}
h_i g_j  (g) &=& h_i (g_j g) \\
& = & (g_j g) g_i^{-1} \\
& = & g_j (g g_i^{-1}) \\
& = & g_j (h_i (g)) \\
& = & g_j h_i (g) \\
\end{array}
$$
We may extend each $h_i \in \bar{h}$ to an automorphism of $X(5,60,\mathbb{Z})$ as follows
$$y \mapsto \left\lbrace\begin{array}{l c}
y & z \in X(5,60,\mathbb{Z})\setminus \bigcup_{i < 60} C_i \\
g_j(y) & y \in C_k \: \textnormal{and} \: g_j: C_k \mapsto h_i(C_k)
\end{array} \right.$$

We now have a $\bar{h} \in \aut(X(5,60,\mathbb{Z}))$ such that $\aut(X(5,60,\mathbb{Z})) \models \comm(\bar{g},\bar{h})$

Remember that $\mathrm{id} \in \bar{g}$ and consider $\bar{g} * \bar{h}$.  For all $g_i h_i \in \bar{g} * \bar{h}$
$$
\begin{array}{rcl}
g_i h_i(\mathrm{id})  &=& g_i \mathrm{id} g_i^{-1} \\
 & = & \mathrm{id}
\end{array}
$$
Since $\mathrm{id}(x)$ was labelled as $\mathrm{id}$, this means that $x \in \supp(\bar{g}) \cap \supp(\bar{h})$, but $x \not\in \supp(\bar{g} * \bar{h})$, contradicting Lemma \ref{lemma:nocancellingorbits}.

\end{proof}

\begin{lemma}\label{30splits}
If $\aut(M) \models A_5(\bar{f})$ and there is an $x \in M$ such that $|\bar{f}(x)|=30$ then there are $\bar{g}$ and $\bar{h}$ such that $\aut(M) \models \comm(\bar{g},\bar{h})$ and $\bar{f} = \bar{g} * \bar{h}$.
\end{lemma}
\begin{proof}
Let $\aut(M) \models A_5(\bar{f})$ be such that there is an $x \in M$ such that $|\bar{f}(x)|=30$.  Let $X$ be the ECC of $\supp(\bar{f})$ that contains $x$.

Let $G$ and $H$ be subgroups of $A_5$ such that $|G|=12$ and $|H|=10$.  There is a transitive action of $A_5$ on $\lbrace aG \: : \: a \in A_5 \rbrace \times \lbrace bH \: : \: b \in A_5 \rbrace$ which is isomorphic (as permutation groups) to $\bar{f}$'s action on $\bar{f}(x)$.

We may therefore label each cone of $X$ as $(aG,bH)$.  We define $\bar{g}, \bar{h} \in \aut(M)$ as follows:
$$
g_i : z \mapsto \left\lbrace
\begin{array}{l c}
f_i(z) & z \in M \setminus X \\
f_j(z) & f_j((aG,bH) ) = (g_iaG,bH)
\end{array}
\right.
$$
$$
h_i : z \mapsto \left\lbrace
\begin{array}{l c}
z & z \in M \setminus X \\
f_j(z) & f_j((aG,bH) ) = (aG,h_ibH)
\end{array}
\right.
$$

If $\bar{g}$ and $\bar{h}$ are not well-defined then there is an $f_i$ such that $f_i((aG,bH))=(aG,bH)$ but there is a $z \in (aG,bH)$ such that $f_i(z) \not= z$.  However $\bar{f}$ acts transitively on the $(aG,bH)$, so $|\bar{f}(z)|= 60$.  Lemma \ref{no60} shows that no such $\bar{f}$ exists.

If $z \in M \setminus X $ then $g_ih_i(z) = g_i(z) = f_i(z)$, so $(\bar{g} * \bar{h})|_{M \setminus X} = \bar{f}|_{M\setminus X}$.  If $z \in (aG,bH)$ then $g_ih_i((aG,bH)) = (g_i(a)G,h_i(b)H) = f_i((aG,bH))$, therefore $g_i h_i(z) = f_i(z)$, and so $(\bar{g}*\bar{h})|_X = \bar{f}|_X$.

Together, we now have $(\bar{g}*\bar{h}) = \bar{f}$, so the lemma is proved.
\end{proof}

\begin{prop}\label{prop:indec}
$\aut(M) \models \indec(\bar{f})$ if and only if $\supp(\bar{f})$ has exactly one extended connected component and every orbit has less than 30 members.
\end{prop}
\begin{proof}
First we prove that if $\supp(\bar{f})$ has exactly one extended connected component and every orbit has less than 30 members then $\aut(M) \models \indec(\bar{f})$ by contradiction.  Let $\bar{g}$ and $\bar{h}$ witness the fact that $\bar{f}$ does not satisfy $\indec$, i.e. $\bar{f} = \bar{g} * \bar{h}$ and $\aut(M) \models \comm(\bar{g},\bar{h})$.

If $\supp(\bar{g}) \cap \supp(\bar{h}) = \emptyset$ then $\bar{f}$ fixes $\supp(\bar{g})$ and $\supp(\bar{h})$ setwise, and hence $\supp(\bar{g})$ and $\supp(\bar{h})$ lie in different ECCs of $\supp(\bar{f})$.

If $\supp(\bar{g}) \cap \supp(\bar{h}) \not= \emptyset$ then Lemma \ref{longorbits} shows that $\bar{g} * \bar{h}$ has an orbit of length at least 20.  If $\bar{g} * \bar{h}$ has an orbit of length 20 then there is also another orbit of length other than 20.  Since the length is other than 20, this other orbit cannot lie in the same ECC as the orbit of length 20.

Therefore if $\supp(\bar{f})$ has exactly one extended connected component and every orbit has less than 30 members then $\aut(M) \models \indec(\bar{f})$.  We now turn our attention to the other direction, which we also do by contradiction.

Suppose $\supp(\bar{f})$ has multiple extended connected components.  We let $X$ be one of these extended connected components and consider $\bar{f}|^{X}$ and $\bar{f}|^{M\setminus X}$.  These two both satisfy $\alt{5}$ (by Lemma \ref{RestrictionSubgroups}) and their supports are disjoint, so they satisfy $\comm$.  Finally $\bar{f}|^{X} * \bar{f}|^{M\setminus X} = \bar{f}$, showing that $\bar{f}|^{X}$ and $\bar{f}|^{M\setminus X}$ witness the fact that $\bar{f}$ does not satisfy $\indec$.

Lemma \ref{no60} shows that $\bar{f}$ cannot have an orbit of length 60.  Lemma \ref{30splits} shows that if $\bar{f}$ has an orbit of length 30 then $\aut(M) \models \neg\indec(\bar{f})$.
\end{proof}

\begin{lemma}\label{lemma:disjbehaves}
If $\aut(M) \models disj(\bar{f},\bar{g})$ then $\supp(\bar{f}) \cap \supp(\bar{g}) = \emptyset$.
\end{lemma}
\begin{proof}
Suppose $\supp(\bar{f}) \cap \supp(\bar{g}) \not= \emptyset$.  By Lemma \ref{noflipping} either
$$\supp(\bar{f}) \subseteq \supp(\bar{g})\:\textnormal{or}\: \supp(\bar{f}) \subseteq \supp(\bar{g})$$
Now assume that $\supp(\bar{f}) \subsetneqq \supp(\bar{g})$ and let $z \in \supp(\bar{f})$.

$\path{\supp(\bar{f}),M \setminus \supp(\bar{f})}$ is a singleton, as $\aut(M) \models \indec(\bar{f})$.  Let
$$\lbrace x_f \rbrace := \path{\supp(\bar{f}),M \setminus \supp(\bar{f})}$$
$x_f \not\in \supp(\bar{f})$, but since $\supp(\bar{f}) \subsetneq \supp(\bar{g})$, we know that $x_f \in \supp(\bar{g})$.

Let $g_i \in \bar{g}$.  By definition $x_f \in \path{g_i^{-1}(x_f),z}$.  Since paths are preserved by automorphisms, this translates to
$$g_i(x_f) \in \path{x_f,g_i(z)}$$
Thus if $g_i \not= \mathrm{id}$ then $g_i(z) \not\in \supp(\bar{f})$, i.e. $f_j g_i(z)= g_i(z)$ for all $j$, but since $z \in \supp(\bar{f})$ there is a $k$ such that $f_k (z) \not=z$.  This is depicted in Figure 25.
$$
\begin{array}{c c l}
g_i(z) & = & f_k g_i (z) \\
 & = & g_i f_k(z) \\
 & \not= & g_i(z)
\end{array}
$$
This is a contradiction.  Therefore if $\supp(\bar{f}) \cap \supp(\bar{g}) \not= \emptyset$ then $\supp(\bar{f}) = \supp(\bar{g})$.

\begin{figure}[ht]\label{fig:disj}
\begin{center}
\begin{tikzpicture}[scale=0.15]
\draw[dotted] (0,20) -- (30,0) -- (60,20);
\fill[white] (30,20) ellipse (30 and 5);
\draw[dotted] (30,20) ellipse (30 and 5);
\fill[dotted] (30,0) circle (0.5);

\fill[white] (45,20) -- (35,30) -- (55,30) -- (45,20);
\draw[dotted] (35,30) -- (45,20) -- (55,30);
\fill[white] (45,30) ellipse (10 and 2);
\draw[dotted] (45,30) ellipse (10 and 2);
\fill (45,20) circle (0.5);

\fill[white] (15,20) -- (5,30) -- (25,30) -- (15,20);
\draw[dotted] (5,30) -- (15,20) -- (25,30);
\fill[white] (15,30) ellipse (10 and 2);
\draw[dotted] (15,30) ellipse (10 and 2);
\fill (15,20) circle (0.5);
\fill (20,30) circle (0.5);

\fill (40,30) circle (0.5);
\fill (50,30) circle (0.5);

\draw[anchor=east] (20,30) node {$g_i(z)$};
\draw[anchor=east] (40,30) node {$z$};
\draw[anchor=east] (50,30) node {$f_k(z)$};

\draw (30,10) node {$\supp(\bar{g})$};
\draw[anchor=north] (45,19) node {$x_f$};
\draw[anchor=north] (15,20) node {$g_i(x_f)$};
\draw (45,26) node {$\supp(\bar{f})$};

\draw[->] (40,30) .. controls (35,35) and (25,35) .. (20.5,30.5); 
\draw[->] (40,30) .. controls (44,34) and (46,34) .. (49.5,30.5);

\draw[anchor=south] (30,33.5) node {$g_i$};
\draw[anchor=south] (45,33) node {$f_k$};
\end{tikzpicture}
\end{center}
\caption{$\supp(\bar{f}) \subsetneq \supp(\bar{g})$}
\end{figure}
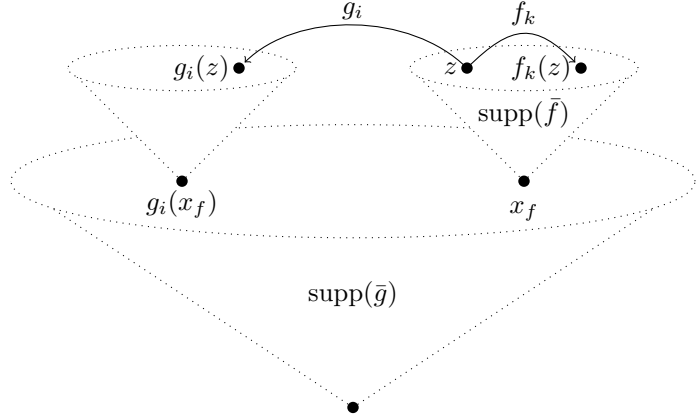

Suppose that $\supp(\bar{f})=\supp(\bar{g})$.  Again, $\path{\supp(\bar{f}),M \setminus \supp(\bar{f})}$ is a singleton, as $\aut(M) \models \indec(\bar{f})$ so we let
$$\lbrace x_f \rbrace := \path{\supp(\bar{f}),M \setminus \supp(\bar{f})}$$

Both $\bar{f}$ and $\bar{g}$ must act transitively on the same antichain of immediate successors or predecessors of $x_f$, which $\bar{f} * \bar{g}$ must also act on.  Since $\aut(M) \models \indec(\bar{f})$ and $\aut(M) \models \indec(\bar{g})$, Proposition \ref{prop:indec} shows that this antichain must have less than 30 members, but Lemma \ref{longorbits} showed that $\bar{f} * \bar{g}$ must have an orbit of at least 20 members.

Lemma \ref{longorbits} also showed that if $\bar{f} * \bar{g}$ has an orbit of length 20 then there was another orbit.  Therefore $\bar{f}$ acts transitively on a set with strictly more than 20 elements, and hence at least 30, which contradicts Proposition \ref{prop:indec}.

Therefore $\supp(\bar{f}) \cap \supp(\bar{g}) = \emptyset$.
\end{proof}

\begin{lemma}\label{FormalSubsetsEq}
Recall that $[\supp(\bar{f}) \sqsubseteq \supp(\bar{g})]$ is the formula
$$
\begin{array}{l c}
\indec(\bar{f}) \wedge \indec(\bar{g}) \wedge \neg \disj(\bar{f},\bar{g}) & \wedge \\
\forall \phi [ \disj( \bar{g}^\phi,\bar{g}) \rightarrow \disj( \bar{f}^\phi , \bar{f} )] & \wedge \\
\forall \phi  ( \bar{g}^\phi \not= \bar{g} \rightarrow \bar{f}^\phi = \bar{f}) & \wedge \\
\end{array}
$$
If $\bar{f}$ and $\bar{g}$ satisfy this formula then the support of $\bar{g}$ is contained in the support of $\bar{f}$.
\end{lemma}
\begin{proof}
The two sentences
$$\forall \bar{f}, \bar{g} \left(
\begin{array}{c}
(\forall \phi [ \disj( \bar{g}^\phi,\bar{g}) \rightarrow \disj( \bar{f}^\phi , \bar{f} )]) \\
 \leftrightarrow \\
 (\neg \exists \phi [\neg \disj( \bar{f}^\phi,\bar{f}) \wedge \disj( \bar{g}^\phi , \bar{g} )])
\end{array}
\right) \;\textnormal{ and } \forall \bar{f}, \bar{g} \left(
\begin{array}{c}
(\forall \phi  ( \bar{g}^\phi \not= \bar{g} \rightarrow \bar{f}^\phi = \bar{f})) \\
 \leftrightarrow \\
(\neg \exists \phi (\bar{f}^\phi = \bar{f} \wedge \bar{g}^\phi \not= \bar{g}))
\end{array}
\right)$$
are tautologies, so the formula given here is equivalent to the one given in Definition \ref{masterdfn}.

Suppose that $\bar{f}$ and $\bar{g}$ are such that
$$\indec(\bar{f}) \wedge \indec(\bar{g}) \wedge \neg \disj(\bar{f},\bar{g})$$
This means that $\supp(\bar{f})$ and $\supp(\bar{g})$ each have exactly one ECC, which have a non-empty intersection.  We define
$$
\begin{array}{c}
x_f := \path{\supp(\bar{f}),M  \setminus \supp(\bar{f})} \\ 
x_g:= \path{\supp(\bar{g}),M \setminus \supp(\bar{g})}
\end{array}
$$

If $\supp(\bar{f}) = \supp(\bar{g})$ then $\supp(\bar{f}^\phi) = \supp(\bar{g}^\phi)$ for all $\phi \in \aut(M)$.   Therefore for all $\phi \in \aut(M)$
$$\aut(M) \models (\disj( \bar{f}^\phi,\bar{f}) \leftrightarrow \disj( \bar{g}^\phi , \bar{g} ))$$
and
$$\aut(M) \models (\bar{g}^\phi \not= \bar{g} \leftrightarrow \bar{f}^\phi = \bar{f} )$$
Thus $\aut(M) \models [\supp(\bar{f}) \sqsubseteq \supp(\bar{g})]$.

We now suppose that $\supp(\bar{f}) \neq \supp(\bar{g})$.  In Case 1 we consider $\supp(\bar{g}) \subsetneq \supp(\bar{f})$.  In Case 3 we consider $\supp(\bar{f}) \subsetneq \supp(\bar{g})$.  If neither $\supp(\bar{g}) \subsetneq \supp(\bar{f})$ nor $\supp(\bar{f}) \subsetneq \supp(\bar{g})$ then we are either in Case 2, where $x_f \not= x_g$, or Case 4 where $x_f=x_g$.

\begin{figure}[h]\label{fig:subsetcases}
\begin{center}
\begin{tikzpicture}[scale=0.05]
\draw (0,40) -- (20,00) -- (40,40);
\draw (10,40) -- (20,20) -- (30,40);
\draw (20,35) node {$\bar{g}$};
\draw (20,10) node {$\bar{f}$};
\draw (20,-5) node {Case $1$};

\fill (20,0) --+ (1,0) --+ (0,1) --+ (-1,0) --+ (0,-1) --+ (1,0);
\fill[white] (20,20) --+ (1,1) --+ (1,-1) --+ (-1,-1) --+ (-1,1) --+ (1,1);
\draw (20,20) + (1,1) --+ (1,-1) --+ (-1,-1) --+ (-1,1) --+ (1,1);

\draw (60,40) -- (80,10) -- (100,40);
\draw (60,0) -- (80,30) -- (100,0);
\draw (80,4) node {$\bar{f}$};
\draw (80,36) node {$\bar{g}$};
\draw (80,-5) node {Case $2$};

\fill (80,30) --+ (1,0) --+ (0,1) --+ (-1,0) --+ (0,-1) --+ (1,0);
\fill[white] (80,10)  --+ (1,1) --+ (1,-1) --+ (-1,-1) --+ (-1,1) --+ (1,1);
\draw (80,10)  + (1,1) --+ (1,-1) --+ (-1,-1) --+ (-1,1) --+ (1,1);

\draw (120,40) -- (140,00) -- (160,40);
\draw (130,40) -- (140,20) -- (150,40);
\draw (140,35) node {$\bar{f}$};
\draw (140,10) node {$\bar{g}$};
\draw (140,-5) node {Case $3$};

\fill (140,20) --+ (1,0) --+ (0,1) --+ (-1,0) --+ (0,-1) --+ (1,0);
\fill[white] (140,0)  --+ (1,1) --+ (1,-1) --+ (-1,-1) --+ (-1,1) --+ (1,1);
\draw (140,0)  + (1,1) --+ (1,-1) --+ (-1,-1) --+ (-1,1) --+ (1,1);

\draw (180,40) -- (200,00) -- (205,40);
\draw[color=gray] (195,40) -- (200,00) -- (220,40);
\draw (190,35) node {$\bar{f}$};
\draw[color=gray] (210,35) node {$\bar{g}$};
\draw (200,-5) node {Case $4$};

\fill[white] (200,0)  + (1,1) --+ (1,-1) --+ (-1,-1) --+ (-1,1) --+ (1,1);
\draw (200,0)  --+ (1,1) --+ (1,-1) --+ (-1,-1) --+ (-1,1) --+ (1,1);
\fill (200,0) --+ (1,0) --+ (0,1) --+ (-1,0) --+ (0,-1) --+ (1,0);

\fill[white] (150,-19)  + (1,1) --+ (1,-1) --+ (-1,-1) --+ (-1,1) --+ (1,1);
\draw (150,-19)  + (1,1) --+ (1,-1) --+ (-1,-1) --+ (-1,1) --+ (1,1);
\draw[anchor=west] (151,-20) node {$=x_g$};

\fill (50,-19) --+ (1,0) --+ (0,1) --+ (-1,0) --+ (0,-1) --+ (1,0);
\draw[anchor=west] (51,-20) node {$=x_f$};
\end{tikzpicture}
\end{center}
\caption{Cases of Lemma \ref{FormalSubsetsEq}}
\end{figure}
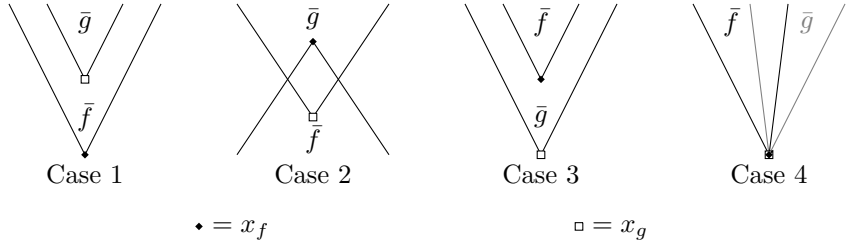

In Case $3$ we must prove that $\aut(M) \models [\supp(\bar{f}) \sqsubseteq \supp(\bar{g})]$, while in Cases $1$ and $2$, we must show that the converse holds.  Finally, in Case 4 we show that $\aut(M) \models [\supp(\bar{f}) \subseteq \supp(\bar{g})]$ if and only if $\supp(\bar{f})=\supp(\bar{g})$.

\paragraph*{Case 1}
Since $x_f$ is moved by $\bar{g}$ there is an $x_f '$ such that $x_f$ and $x_f'$ lie in the same $\bar{g}$-orbit and $x_f \not= x_f'$.  Let $\phi$ be an automorphism that switches $x_f$ and $x_f'$, but fixes anything that it does not have to move.  If $z \in \supp(\bar{f})$ then $\phi(z) \not\in \supp(\bar{f})$ and so $\disj(\bar{f}^\phi,\bar{f})$.  Since $\path{x_f,x_f'} \subseteq \supp(\bar{g})$ we know that $\supp(\bar{g}) = \supp(\bar{g}^\phi)$ and therefore $\neg \disj(\bar{g}^\phi, \bar{g})$

Thus $\phi$ witnesses the fact that $\bar{f}$ and $\bar{g}$ do not satisfy $[\supp(\bar{g}) \sqsubseteq \supp(\bar{f})]$.

\paragraph*{Case 2}
Let $x_f'$ be such that $x_f \in \path{x_g,x_f'}$ and $x_f \parallel x_f'$.  Since $X(n,m, \mathbb{Z})$ is 1-transitive there is an automorphism $\phi$ such that $\phi(x_f)=x_f'$.  We know that $\disj(\bar{f}^\phi,\bar{f})$ as
$$\path{\supp(\phi*\bar{f}),\supp(\bar{f})}=\path{f,f'}$$
which cannot be empty, as $x_f \parallel x_f'$.  Since $x_f \in \path{x_g,x_f'}$ and $x_f \in \supp(\bar{g})$ the support of $\bar{g}$ must contain $x_f'$.  However $x_f'$ is clearly contained in $\supp(\phi*\bar{g})$, so $\neg \disj(\phi*\bar{g}, \bar{g})$.

Thus $\phi$ witnesses the fact that $\bar{f}$ and $\bar{g}$ do not satisfy $[\supp(\bar{g}) \sqsubseteq \supp(\bar{f})]$.

\paragraph*{Case 3}
For a contradiction, assume that
$$\aut(M) \models \exists \phi [\disj( \bar{f}^\phi ,\bar{f}) \wedge \neg\disj( \bar{g}^\phi , \bar{g} )]$$
and let $\phi$ witness this.  Since $\disj( \bar{f}^\phi,\bar{f})$ holds, and $\supp(\bar{g})$ is contained in $\supp(\bar{f})$, we know that $\disj( \bar{g}^\phi,\bar{g})$, giving a contradiction.

Now assume that 
$$\aut(M) \models \exists \phi (\bar{f}^\phi = \bar{f} \wedge \bar{g}^\phi \not= \bar{g}) $$
Let $C_0, C_1$ be two of the cones of $x_f$ that are contained in the support of $\bar{f}$ and let $f_i \in \bar{f}$ map $C_0$ to $C_1$.  Since $\bar{g}^\phi \not= \bar{g}$, there is an $x \in \supp(\bar{g})$ such that $\phi(x) \not= x$.  We suppose without loss of generality that $x \in C_0$.

If $\phi(x) \not\in C_1$ then $f_i^\phi$ will map $x$ to $f_i\phi (x) \not= f_i(x)$ and so $\bar{f}^\phi \not= \bar{f}$.  If $\phi(x) \in C_1$ then conjugation by $\phi$ will at least switch the roles $C_0$ and $C_1$, and so $\bar{f}^\phi \not= \bar{f}$.

\paragraph*{Case 4}
In this case, $x_f=x_g$.  Let $C^f_0, \ldots $ be the cones of $f$ that are contained in $\bar{f}$, and let $C^g_0, \ldots $ be the cones of $x_g$ that are contained in $\bar{g}$.  We may pick our indices such that $C^f_i \in \supp(\bar{f}) \cap \supp(\bar{g})$ if and only if $C^g_i \in \supp(\bar{f}) \cap \supp(\bar{g})$.

Assume that only one $C^f_i$ is not in the intersection of the supports, and assume without loss of generality that this is $C^f_0$.  Let $\phi \in \aut(M)$ be such that $\supp(\phi) \subsetneq C^g_0$ and.  Then
$$\aut(M) \models (\bar{f}^\phi = \bar{f} \wedge \bar{g}^\phi \not= \bar{g})$$
showing that $\bar{f}$ and $\bar{g}$ do not satisfy $[\supp(\bar{f}) \sqsubseteq \supp(\bar{g})]$.

Now we assume that more that one $C^f_i$ is not in the intersection of the supports, without loss of generality $C^f_0$ and $C^f_1$.  Let $\phi \in \aut(M)$ be such that $\phi$ swaps $C^g_0$ and $C^g_1$ and fixes everything else point-wise.  Since $\phi$ fixes $\supp(\bar{f})$ point-wise, $\aut(M) \models \bar{f}^\phi = \bar{f}$.

Now consider a elements of $\bar{g}$ which switches $C^g_0$ and $C^G_2$.  The corresponding elements of $\bar{g}^\phi$ will switch $C^g_1$ and $C^g_2$, and so $\aut(M) \models \bar{g}^\phi \not= \bar{g}$.
\end{proof}

\begin{cor}\label{FormalSubsets}
Recall that $[\supp(\bar{g}) \sqsubset \supp(\bar{f})]$ is the formula
$$[\supp(\bar{g}) \sqsubseteq \supp(\bar{f})] \wedge \neg [\supp(\bar{f}) \sqsubseteq \supp(\bar{g})]$$
$\aut(M) \models [\supp(\bar{g}) \sqsubset \supp(\bar{f})]$ if and only if $\supp(\bar{g})$ is properly contained in $\supp(\bar{f})$.
\end{cor}

\begin{dfn}
Let $\aut(M) \models \indec(\bar{f}) \wedge \indec(\bar{g})$ and let
$$
\begin{array}{c}
x_f:= \path{\supp(\bar{f}),M\setminus \supp(\bar{f})} \\
x_g:= \path{\supp(\bar{g}),M \setminus \supp(\bar{g})}
\end{array}
$$
We say that $\bar{f}$ and $\bar{g}$ have the \textbf{same direction}, or act in the same direction if
$$\exists y \in \supp(\bar{f}) \: (x_f < y) \Leftrightarrow \exists z \in \supp(\bar{g}) \: ( x_g < z )$$
We say that $\bar{f}$ and $\bar{g}$ have \textbf{different directions}, or act in different directions if
$$\exists y \in \supp(\bar{f}) \: x_f < y \Leftrightarrow \exists z \in \supp(\bar{g}) \: (x_g > z)$$
\end{dfn}

\begin{lemma}\label{SamePDBehaves}
Recall that $\mathrm{SamePD}(\bar{f},\bar{g})$ is the formula
$$\forall \bar{h} ( [\supp(\bar{h}) \sqsubset \supp(\bar{f})] \leftrightarrow [\supp(\bar{h}) \sqsubset \supp(\bar{g})])$$
Let
$$
\begin{array}{c}
\lbrace x_f \rbrace:= \path{\supp(\bar{f}),M\setminus \supp(\bar{f})} \\
\lbrace x_g \rbrace:= \path{\supp(\bar{g}),M \setminus \supp(\bar{g})}
\end{array}
$$
If
$$\aut(M) \models \mathrm{SamePD}(\bar{f},\bar{g})$$
then $f=g$ and $\bar{f}$ and $\bar{g}$ have the same direction.
\end{lemma}
\begin{proof}
Suppose $\aut(M) \models \mathrm{SamePD}(\bar{f},\bar{g})$

We will first show that $x_f=x_g$ by contradiction.  Suppose that $x_g \in \supp(\bar{f})$.  If $\supp(\bar{g}) \subset \supp(\bar{f})$ then $\bar{f}$ witnesses that $\bar{f}$ and $\bar{g}$ cannot satisfy $\mathrm{SamePD}(\bar{f},\bar{g})$.  If $\supp(\bar{g}) \not\subset \supp(\bar{f})$ then the supports of $\bar{f}$ and $\bar{g}$ are as in the pictures in Case 2 of Figure 26.

Let $\bar{h}$ be a tuple such that:
\begin{enumerate}
\item $\aut(M) \models \indec(\bar{h})$;
\item $\lbrace x_f \rbrace = \path{\supp(\bar{h}),M\setminus \supp(\bar{h})}$; and
\item $\bar{f}$ and $\bar{g}$ act in different directions.
\end{enumerate}
Then $\supp(\bar{h}) \subset \supp(\bar{g})$ and $\supp(\bar{h}) \cap \supp(\bar{f}) = \emptyset$, giving a contradiction. 

Now suppose that $x_g \not\in \supp(\bar{f})$ and $x_f \not\in \supp(\bar{g})$.  We consider two situations, where the point of $\path{x_f,x_g}$ next to $x_f$ is in the same direction as $\bar{f}$ or in the other direction (depicted in Figure 27).

\begin{figure}[h]\label{directionofpath}
\begin{center}
\begin{tikzpicture}[scale=0.12]
\draw (0,20) -- (10,0) -- (20,20);
\draw (25,10) -- (10,0) -- (25,-10);

\fill[white] (10,20) ellipse (10 and 3);
\draw (10,20) ellipse (10 and 3);

\fill (10,0) circle (0.3);
\fill (25,10) circle (0.3);
\fill (25,-10) circle (0.3);

\draw (10,20) node {$\bar{f}$};
\draw (6,0) node {$x_f$};
\draw (27,10) node {$x_1$};
\draw (27,-10) node {$x_2$};
\end{tikzpicture}
\end{center}
\caption{$\path{f,g}$ and the Direction of $\bar{f}$}
\end{figure}
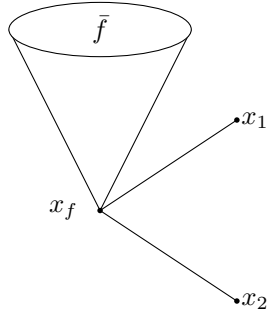

This picture depicts both situations.  By ``the point of $\path{x_f,x_g}$ immediate to $f$ is in the same direction as $\bar{f}$'' we mean that $x_1 \in \path{f,g}$, while $x_2 \in \path{f,g}$ is the other situation we need to consider.

Suppose $x_1 \in \path{f,g}$ and let $\phi$ be an automorphism of $M$ which fixes $f$ and switches $x_1$ with a member of $\supp(\bar{f})$.  Then $\phi*\bar{f}$ witnesses the fact that $\bar{f}$ and $\bar{g}$ cannot satisfy $\mathrm{SamePD}(\bar{f},\bar{g})$.  If $x_2 \in \path{f,g}$ then any tuple that satisfies $\indec$, fixes $f$ and moves $x_2$ will do as a witness.

We know that if $\aut(M) \models \mathrm{SamePD}(\bar{f},\bar{g})$ then $x_f=x_g$.  If $\bar{f}$ and $\bar{g}$ act in different directions then we may pick any point in $\supp(\bar{g})$ and any tuple that fixes that point and moves $x_f$ to find our counter-example.

It remains to show that if $\bar{f}$ and $\bar{g}$ fix the same point and have the same direction then they satisfy $\mathrm{SamePD}$.  Assume without loss of generality that $\bar{f}$ and $\bar{g}$ act on the successors of $x_f$.  Let $\bar{h}$ be any tuple such that
$$[\supp(\bar{f}) \sqsubset \supp(\bar{h})]$$
This means that $\bar{h}$ moves $x_f$ and all its successors, and therefore $\supp(\bar{g})$ contains the support of $\bar{g}$, and so $\bar{h}$ satisfies $[\supp(\bar{f}) \sqsubset \supp(\bar{h})]$.
\end{proof}

\begin{lemma}
Recall that $\mathrm{RepPoint}(\bar{f}_0,\bar{f}_1)$ is the formula
$$
\disj(\bar{f}_0,\bar{f}_1) \wedge \forall \bar{g} \exists \bar{h} ( \neg \disj(\bar{g},\bar{h}) \wedge ( \mathrm{SamePD}(\bar{f_0},\bar{h})  \vee \mathrm{SamePD}(\bar{f_1},\bar{h}) ) )
$$
Let
$$
\begin{array}{c}
\lbrace x_0 \rbrace:= \path{\supp(\bar{f}_0),M) \setminus \supp(\bar{f}_0)} \\
\lbrace x_1 \rbrace:= \path{\supp(\bar{f}_1),M) \setminus \supp(\bar{f}_1)}
\end{array}
$$
Then
$$\aut(M) \models \mathrm{RepPoint}(\bar{f}_0,\bar{f}_1)$$
if and only if $x_0= x_1$ and $\bar{f}_0$ and $\bar{f}_1$ act in different directions.
\end{lemma}
\begin{proof}
First we will prove that if $\bar{f}_0$ and $\bar{f}_1$ are such that $x_0= x_1$ and $\bar{f}_0$ and $\bar{f}_1$ act in different directions then
$$\aut(M) \models \mathrm{RepPoint}(\bar{f}_0,\bar{f}_1)$$
If $\neg \disj(\bar{g},\bar{f}_0)$ or $\neg \disj(\bar{g},\bar{f}_1)$ then we may take $\bar{h} = \bar{f}_0$ or $\bar{h}= \bar{f}_1$, so suppose that $\disj(\bar{g},\bar{f}_0)$ and $\disj(\bar{g},\bar{f}_1)$.

Let
$$\lbrace x_g \rbrace:= \path{\supp(\bar{g}),M) \setminus \supp(\bar{g})}$$
and let $\bar{h}$ be such that
$$( \mathrm{SamePD}(\bar{f_0},\bar{h})  \vee \mathrm{SamePD}(\bar{f_1},\bar{h}) ) $$
and $\path{x_0,x_g} \subset \supp(\bar{h})$.  Clearly this $\bar{h}$ is as required by the formula.

Now we must prove that if
$$\aut(M) \models \mathrm{RepPoint}(\bar{f}_0,\bar{f}_1)$$
then $\bar{f}_0$ and $\bar{f}_1$ are as desired.  If $x_0 \not= x_1$ then there is some $y$ such that none of the following hold
$$
\begin{array}{c c c}
y \in \path{x_0,x_1} & x_0 \in \path{y,x_1} & x_1 \in \path{y,x_0}
\end{array}
$$
Let $\bar{g}$ be such that $\path{y,\lbrace x_0,x_1 \rbrace} \not\subset \supp(\bar{g})$ and
$$y = \path{\supp(\bar{g}),M \setminus \supp(\bar{g})}$$
This $\bar{g}$ witnesses the fact that $\bar{f}_0$ and $\bar{f}_1$ do not satisfy $\mathrm{RepPoint}(\bar{f}_0,\bar{f}_1)$.

Now suppose that $x_0 = x_1$ but
$$\aut(M) \models \mathrm{SamePD}(\bar{f}_0,\bar{f}_1)$$
In this case any $\bar{g}$ whose support is disjoint from that of $\bar{f}_0$ and $\bar{f}_1$ and which fixes $f_0$ will be a witness.
\end{proof}

We now have our formula that defines the domain of interpretation, however there will be a lot of pairs that satisfy $\mathrm{RepPoint}$ but fix the same point.  

\begin{lemma}
Recall that $\mathrm{EqRepPoint}(\bar{f}_0,\bar{f}_1;\bar{g}_0,\bar{g}_1)$ is the formula
$$
\begin{array}{c}
\mathrm{RepPoint}(\bar{f}_0,\bar{f}_1) \wedge \mathrm{RepPoint}(\bar{g}_0,\bar{g}_1) \wedge \\
(\mathrm{SamePD}(\bar{f}_0,\bar{g}_0) \wedge \mathrm{SamePD}(\bar{f}_1,\bar{g}_1)) \vee (\mathrm{SamePD}(\bar{f}_0,\bar{g}_1) \wedge \mathrm{SamePD}(\bar{f}_1,\bar{g}_0))
\end{array}
$$
Let
$$
\begin{array}{c}
x_f:= \path{\supp(\bar{f}_0),M)\setminus \supp(\bar{f}_0)} \\
x_g:= \path{\supp(\bar{g}_0),M) \setminus \supp(\bar{g}_0)}
\end{array}
$$
If
$$\aut(M) \models \mathrm{EqRepPoint}(\bar{f}_0,\bar{f}_1;\bar{g}_0,\bar{g}_1)$$
then $x_f=x_g$.
\end{lemma}
\begin{proof}
Clearly $x_f \not= x_g$ if and only if $\mathrm{SamePD}(\bar{f}_i,\bar{g}_j)$ holds for some choice of indices.
\end{proof}

\subsection{Interpreting Betweenness}\label{interpretingbetweenness}

From now on we will adopt the convention that when a lower case letter, such as $g$, appears in one of our formulas, it is actually a pair $(\bar{g}_0,\bar{g}_1)$ that satisfies $\mathrm{RepPoint}$.  We will refer to the point represented by $g$ as $x_g$.

\begin{dfn}
$\mathrm{Temp1PB}(g;h,k)$ is the following formula:
$$\exists l (\mathrm{EqRepPoint}(\bar{g_0},\bar{g_1};\bar{l_0},\bar{l_1}) \wedge \left(
\begin{array}{l}
\neg \disj(\bar{l_0},\bar{h_0}) \wedge \neg \disj(\bar{l_0},\bar{h_1}) \wedge \\
\neg \disj(\bar{l_1},\bar{k_0}) \wedge \neg \disj(\bar{l_1},\bar{k_1})
\end{array} \right)
$$
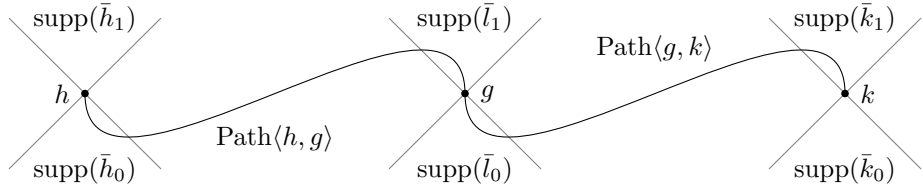
\begin{figure}[h]
\begin{center}
\begin{tikzpicture}[scale=0.1]

\draw[gray] (-60,10) -- (-40,-10);
\draw[gray] (-60,-10) -- (-40,10);

\draw[gray] (-10,10) -- (10,-10);
\draw[gray] (-10,-10) -- (10,10);

\draw[gray] (60,10) -- (40,-10);
\draw[gray] (60,-10) -- (40,10);

\draw (-50,0) .. controls (-50,-20) and (0,20) .. (0,0);
\draw (50,0) .. controls (50,20) and (0,-20) .. (0,0);

\fill (-50,0) circle (0.5);
\fill (0,0) circle (0.5);
\fill (50,0) circle (0.5);

\draw (-50,10) node {$\supp(\bar{h}_1)$};
\draw (-50,-10) node {$\supp(\bar{h}_0)$};
\draw (0,10) node {$\supp(\bar{l}_1)$};
\draw (0,-10) node {$\supp(\bar{l}_0)$};
\draw (50,10) node {$\supp(\bar{k}_1)$};
\draw (50,-10) node {$\supp(\bar{k}_0)$};

\draw (3,0) node {$g$};
\draw (-53,0) node {$h$};
\draw (53,0) node {$k$};

\draw (-25,-6) node {$\path{h,g}$};
\draw (25,6) node {$\path{g,k}$};
\end{tikzpicture}
\end{center}
\caption{What is described by $\mathrm{Temp1PB}(g;h,k)$}
\end{figure}

$\mathrm{Temp2PB}(g;h,k)$ is the formula
$$\phi(g;h,k) \wedge \forall l \, \phi(l;h,k) \rightarrow \left[
\begin{array}{l r}
\mathrm{Temp1PB}(g;l,k) & \wedge \\
\mathrm{Temp1PB}(g;l,h)
\end{array} \right]$$
where $\phi$ is the formula that requires, using $\mathrm{disj}$, the configurations of the supports of $\bar{g}_0$, $\bar{g}_1$, $\bar{h}_0$, $\bar{h}_1$, $\bar{k}_0$ and $\bar{k}_1$ depicted in Figure 29, for all permutations of the indices and that each pair represents different points.
\begin{figure}[h]\label{fig:temp2}
\begin{center}
\begin{tikzpicture}[scale=0.11]

\draw[gray] (-60,10) -- (-40,-10);
\draw[gray] (-60,-10) -- (-40,10);

\draw[gray] (-10,10) -- (10,-10);
\draw[gray] (-10,-10) -- (10,10);

\draw[gray] (60,10) -- (40,-10);
\draw[gray] (60,-10) -- (40,10);

\draw[gray] (10,-20) -- (30,-40);
\draw[gray] (10,-40) -- (30,-20);

\draw (-50,0) .. controls (-50,-20) and (-15,10) .. (-5,10);
\draw (50,0) .. controls (50,-20) and (15,10) .. (5,10);
\draw (0,0) .. controls (0,-20) and (20,0) .. (20,-30);
\draw (-5,10) -- (0,0) -- (5,10);

\fill (-50,0) circle (0.5);
\fill (0,0) circle (0.5);
\fill (20,-30) circle (0.5);
\fill (50,0) circle (0.5);
\fill (-5,10) circle (0.5);
\fill (5,10) circle (0.5);

\draw (-50,10) node {$\supp(\bar{h}_1)$};
\draw (-50,-10) node {$\supp(\bar{h}_0)$};
\draw (0,15) node {$\supp(\bar{g}_1)$};
\draw (-6.5,-11) node {$\supp(\bar{g}_0)$};
\draw (50,10) node {$\supp(\bar{k}_1)$};
\draw (50,-10) node {$\supp(\bar{k}_0)$};
\draw (30,-20) node {$\supp(\bar{l}_1)$};
\draw (20,-40) node {$\supp(\bar{l}_0)$};

\draw (3,0) node {$x_g$};
\draw (-53,0) node {$x_h$};
\draw (53,0) node {$x_k$};
\draw (23,-30) node {$x_l$};

\draw (25,-12.5) node {$\path{x_g,x_l}$};
\draw (-30,5) node {$\path{x_h,x_g}$};
\draw (30,5) node {$\path{x_g,x_k}$};
\draw (-25,-25) node {All suitable $x_l$ occur in $\supp(\bar{g}_0)$};
\end{tikzpicture}
\end{center}
\caption{What is described by $\mathrm{Temp2PB}(g;h,k)$}
\end{figure}
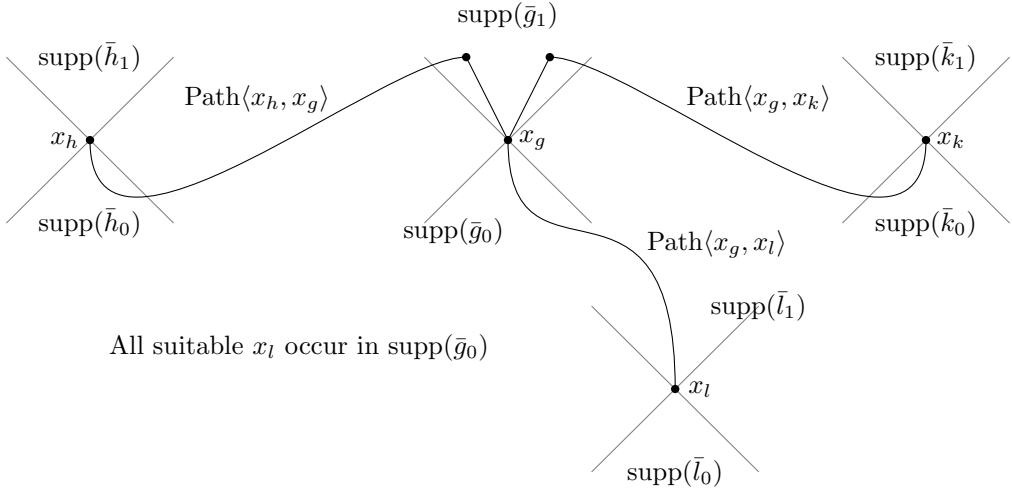

$\mathrm{PathBetween}(g;h,k)$ is the formula
$$
\mathrm{Temp1PB}(g;h,k) \vee \mathrm{Temp2PB}(g;h,k)
$$
\end{dfn}

\begin{lemma}
The previously defined formulas express the following properties of the structure:
\begin{enumerate}
\item $\mathrm{Temp1PB}(g;h,k)$ holds if and only if $\path{x_h,x_k}$ contains a chain of length at least three, of which $x_g$ is one of the middle points.
\item $\mathrm{Temp2PB}(g;h,k)$ holds only if $x_g$ is a local maximum or minimum of $\path{x_h,x_k}$.
\item $\mathrm{PathBetween}(x_g;x_h,x_k)$ holds if and only if $x_g \in \path{x_h,x_k}$.
\end{enumerate}
\end{lemma}
\begin{proof}
Without loss of generality, we suppose that the situation is the same as depicted in the diagrams above.
\begin{enumerate}
\item Since the formula $\mathrm{Temp1PB}$ insists that $x_h \in \supp(\bar{g}_1)$ and $x_k \in \supp(\bar{g}_0)$, and since any path between something in $\supp(\bar{g}_0)$ and something in $\supp(\bar{g}_1)$ must pass through $x_g$, we conclude that $x_g \in \path{x_h,x_k}$.  Additionally, since $\bar{g}_0$ and $\bar{g}_1$ point in different directions there must be both an immediate successor and an immediate predecessor of $x_g$ lying on $\path{x_h,x_k}$ thus showing that if $\mathrm{Temp1PB}$ holds then the properties it was intended to describe hold.  The other direction is immediate.
\item Since the formula $\mathrm{Temp2PB}$ holds both $x_h$ and $x_k$ are in $\supp(\bar{g}_1)$, if $x_g \in \path{x_h,x_k}$ then it is either a local maximum or a local minimum, as $\supp(\bar{g}_0)$ is an extended connected component originating at $x_g$.  If $x_g \not\in \path{x_h,x_k}$ then
$$\lbrace x_g \rbrace \subsetneq \path{x_g,x_h} \cap \path{x_g,x_k}$$
Any $$x_l \in [ \path{x_g,x_h} \cap \path{x_g,x_k} ]  \setminus \lbrace x_g \rbrace$$ will prevent $\mathrm{Temp2PB}$ from holding.  Again, the other direction is immediate.
\item If $x_g \in \path{x_h,x_k}$ then either $x_g$ is a local maximum or minimum, or $x_g$ lies on a chain of length at least 3, so $\mathrm{Temp1PB}$ and $\mathrm{Temp2PB}$ successfully cover every case.
\end{enumerate}
\end{proof}

\begin{dfn}
$\mathrm{Related}(f,g)$ is the formula
$$\forall h (\mathrm{PathBetween}(h;f,g) \rightarrow \mathrm{Temp1PB}(h;f,g))$$
\end{dfn}

\begin{lemma}
$\mathrm{Related}(f,g)$ holds if and only if $x_f \leq x_g$ or $x_g < x_f$.
\end{lemma}

At this point we have recovered $M$ up to order reversal.  We may, if we wish, recover the full order using a variety of different methods, which I will detail later, but from here we can prove that the class is faithful by recovering the betweenness reduct of the CFPOs in question.

\begin{dfn}
$B(h;f,g)$ is the formula

$$\mathrm{PathBetween}(h;f,g) \wedge \left(
\begin{array}{l r}
\mathrm{Related}(f,g) & \wedge \\
\mathrm{Related}(f,h) & \wedge \\
\mathrm{Related}(g,h)
\end{array} \right)$$
\end{dfn}

\begin{lemma}
$B(h;f,g)$ if and only if $x_h$ is between $x_f$ and $x_g$.
\end{lemma}

\begin{theorem}
$K_{Cone}$ is faithful.
\end{theorem}
\begin{proof}
Let $\langle M, \leq \rangle, \langle N, \leq \rangle \in K_{Cone}$.  Let $\Phi$ be the first-order interpretation comprising of:
\begin{itemize}
\item $\mathrm{RepPoint}(x)$ as the formula that defines the domain of interpretation;
\item $\mathrm{EqRepPoint}(x,y)$ as the equivalence relation on the domain of interpretation;
\item $B(z;x,y)$ as the betweenness relation.
\end{itemize}
We have established previously that for all $M$
$$\Phi(\aut(\langle M, \leq \rangle)) \cong \langle M, B \rangle$$
Therefore $\aut \langle M, \leq \rangle \cong \aut \langle N, \leq \rangle$ if and only if $\langle M, B \rangle \cong \langle N, B \rangle$.

If $\langle M, B \rangle \cong \langle N, B \rangle$ then $\langle M, \leq \rangle \cong \langle N, \leq \rangle$ or $\langle M, \leq \rangle \cong \langle N, \leq^* \rangle$ (the reverse ordering).  By assumption, this means that $\langle M, \leq \rangle \cong \langle N, \leq \rangle$, thus the class is faithful.
\end{proof}

\section{Reconstructing the Order}

It is impossible to reconstruct the order of all members of $K_{Cone}$ with a first-order interpretation.  Some members of $K_{Cone}$ are isomorphic to their own reverse image, so the automorphism group has no way of knowing which way is `up'.

In those circumstances, it will be necessary to make an artificial choice over which way is `up'.  When reconstructing linear orders in \cite{RubinOrder}, McCleary and Rubin use a parameter pair for this purpose, obtaining a formula $\phi(x_1,x_2;y_1,y_2)$, which interprets
$$x_1 \leq x_2 \Leftrightarrow y_1 \leq y_2$$
This approach is also possible in this context, but not in a first order way.

Since all members of $K_{Cone}$ embed the alternating chain, as the path between $\lbrace x_1, x_2 \rbrace$ and $\lbrace y_1,y_2 \rbrace$ grows, we require longer and longer formulas.  We must use an $L_{\omega_1,\omega}$ formula to recover the order with this technique.

Another approach would be to exploit the fact that we have insisted that
$$ro \downarrow (M) \leq ro \uparrow (M)$$
Ramification order is definable when finite, so if $ ro\downarrow(M) < \lbrace ro\uparrow(M), \aleph_0 \rbrace$, then we can find a first order formula that depends on $ro \downarrow$ that interprets the order.

While first order, I find this far less satisfactory, as it gives lots of different formulas, each of which only work in limited circumstances.  Even together they do not work everywhere.  However, I will present both.

\subsection{$ ro\downarrow (M) < \lbrace ro\uparrow (M), \aleph_0 \rbrace$}

\begin{dfn}
Let $K^n_{Cone}:= \lbrace M \in K_{Cone} \: : \: ro \downarrow (M) \leq n < ro\uparrow (M) \rbrace $.
\end{dfn}

\begin{dfn}
$ x \lessdot_n y $ is the following formula
$$\mathrm{Related}(x,y) \wedge \exists x_0, \ldots, x_n \left(
\begin{array}{c l r}
\bigwedge\limits_{0\leq i \leq n}& \mathrm{Related}(x,x_i) & \wedge \\
\bigwedge\limits_{i\not= j}& \neg\mathrm{Related}(x_i,x_j) & \wedge \\
\bigwedge\limits_{i\not= j}& \mathrm{PathBetween}(x;x_i,x_j) & \wedge \\
& \neg \mathrm{PathBetween}(x;y,x_0)
\end{array}
\right)$$
\end{dfn}

\begin{theorem}
If $M \in K^n_{Cone}$ then $\aut(M) \models x \lessdot_n y$ if and only if $M \models x <_M y$.
\end{theorem}
\begin{proof}
By definition, $x \lessdot_n y \rightarrow \mathrm{Related}(x,y)$, so if $\aut(M) \models x \lessdot_n y$ then $M \models x \leq \geq_M y$.

Each of the $x_i$ are related to $x$, but $\lbrace x_i \: : \: i=0, ..., n \rbrace$ forms an antichain.  Suppose that none of the $x_i$'s lie above $x$.  Since $ro \downarrow (M) \leq n$ this means that at least two of the $x_i$'s, say $x_0$ and $x_1$, are contained in the same downwards cone of $x$.

Therefore $x_0 \vee x_1 < x$, but the connecting set of the path from $x_0$ to $x_1$ must be
$$\lbrace x_0, x_0 \vee x_1, x_1 \rbrace$$
which would imply that $x \not\in \path{x_0,x_1}$, which contradicts the assumption that $\aut(M) \models x \lessdot_n y$.  Thus at least one of the $x_i$'s is above $x$.

Suppose, without loss of generality, that $x_0$ is above $x$.  If any of the other $x_i$'s lie below $x_0$ then they will be related to $x_i$, giving a contradiction.  By the above argument, all of the $x_i$'s lie in different cones.

On the other hand, any $n+1$ element antichain above $x$, where every element is contained in a different cone above $x$ satisfies the all of the properties demanded of it, except
$$(\bigvee\limits_{i \leq n} \neg \mathrm{PathBetween}(x;y,x_i))$$
If $x < y$ then we will be able to choose $x_0$ such that $x_0$ is contained in the same cone as $y$, so any such antichain will satisfy the formula.

If $y<x$ then any path from any of the $x_i$'s to $y$ will pass through $x$, and so the formula cannot be satisfied.
\end{proof}

\subsection{Abandoning First Order Logic}

Throughout this subsection, we assume that $y_1$ and $y_2$ satisfy $\mathrm{Related}$.  All the formulas mentioned will use $y_1$ and $y_2$ as parameters.  We will use $y_1$ and $y_2$ to indicate the direction of the order, so we suppose that $y_1 < y_2$.

\begin{dfn}
$  (x_1 <_0 x_2 \Leftrightarrow y_1 < y_2)$ is the formula that insists that $x_1$, $x_2$, $y_1$ and $y_2$ are all related and using $B(z;x,y)$ insists that they lie in one of the configurations depicted below.
\begin{figure}[h]
\begin{center}
\begin{tikzpicture}[scale=0.065]
\draw[gray] (10,100) -- (10,90);
\draw[gray] (20,100) -- (20,80);
\draw[gray] (30,100) -- (30,65);
\draw[gray] (40,100) -- (40,40);
\draw[gray] (50,100) -- (50,20);
\draw[gray] (60,80) -- (60,65);
\draw[gray] (70,80) -- (70,40);
\draw[gray] (80,80) -- (80,20);
\draw[gray] (90,65) -- (90,52.5);
\draw[gray] (100,65) -- (100,40);
\draw[gray] (110,65) -- (110,20);
\draw[gray] (120,40) -- (120,20);
\draw[gray] (130,30) -- (130,20);

\draw (0,40) -- (140,40);
\draw (0,80) -- (140,80);

\fill (10,100) circle (1);
\fill (9,89) rectangle (11,91);
\fill (20,100) circle (1);
\fill (19,79) rectangle (21,81);
\fill (30,100) circle (1);
\fill (29,64) rectangle (31,66);
\fill (40,100) circle (1);
\fill (39,39) rectangle (41,41);
\fill (50,100) circle (1);
\fill (49,19) rectangle (51,21);
\fill (60,80) circle (1);
\fill (59,64) rectangle (61,66);
\fill (70,80) circle (1);
\fill (69,39) rectangle (71,41);
\fill (80,80) circle (1);
\fill (79,19) rectangle (81,21);
\fill (90,65) circle (1);
\fill (89,51.5) rectangle (91,53.5);
\fill (100,65) circle (1);
\fill (99,39) rectangle (101,41);
\fill (110,65) circle (1);
\fill (109,19) rectangle (111,21);
\fill (120,40) circle (1);
\fill (119,19) rectangle (121,21);
\fill (130,30) circle (1);
\fill (129,19) rectangle (131,21);

\fill (99,89) rectangle (101,91);
\fill (100,100) circle (1);
\draw[anchor=west] (100,100) node {$=x_2$};
\draw[anchor=west] (100,90) node {$=x_1$};

\draw (130,77) node {$y_2$};
\draw (10,37) node {$y_1$};
\end{tikzpicture}
\end{center}
\caption{Defining $ (x_1 <_0 x_2 \Leftrightarrow y_1 < y_2)$}
\end{figure}
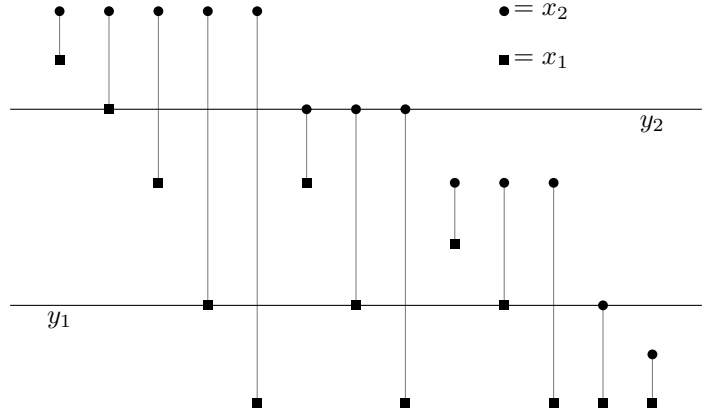
\end{dfn}

\begin{lemma}\label{lemma:0}
If $\aut(M) \models (x_1 <_0 x_2 \Leftrightarrow y_1 < y_2)$ then $M \models x_1 <_M x_2$.
\end{lemma}
\begin{proof}
All possible cases are covered by the definition.
\end{proof}

\begin{dfn}
$(x_1 <_1 x_2 \Leftrightarrow y_1 < y_2)$ is the formula
$$
\neg (x_2 <_0 x_1 \Leftrightarrow y_1 < y_2) \wedge \neg (x_1 <_0 x_2 \Leftrightarrow y_1 < y_2) \wedge ( \alpha_1 \vee \alpha_2 \vee \alpha_3 \vee \alpha_4)
$$
where:
$$
\begin{array}{r l}
\alpha_1 := &B (y_2;y_1,x_2) \wedge \mathrm{Related}(x_1,x_2)\\
\alpha_2 := &B(x_2; x_1,y_2)\\
\alpha_3 := &B(y_1;x_1,y_2) \wedge \mathrm{Related}(x_1,x_2)\\
\alpha_4 := &B(x_1; y_1,x_2) \\
\alpha_5 := & \bigwedge\limits_{i \not= j}\mathrm{Related}(x_i,y_j) \wedge \bigwedge\limits_{i = j} \neg \mathrm{Related}(x_i,y_j) \wedge \mathrm{Related}(x_1,x_2)

\end{array}
$$
\end{dfn}

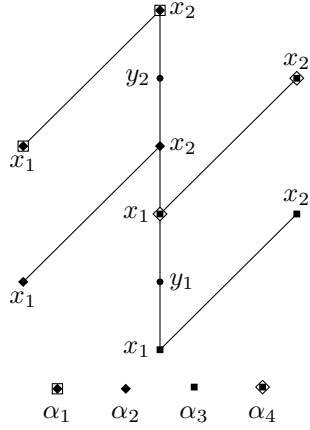
\begin{figure}[h]
\begin{center}
\begin{tikzpicture}[scale=0.09]
\draw (0,30) -- (20,50) -- (20,0) -- (40,20);
\draw (40,40) -- (20,20);
\draw (0,10) -- (20,30);

\fill[anchor=north] (20,0.5) --++ (0.5,0) --++ (0,-1) --++ (-1,0) --++ (0,1) --++ (0.5,0);
\fill[anchor=north] (40,20.5) --++ (0.5,0) --++ (0,-1) --++ (-1,0) --++ (0,1) --++ (0.5,0);

\fill (20,10) circle (0.5);

\fill[white] (20,21) --++ (1,-1) --++ (-1,-1) --++ (-1,1) --++ (1,1);
\fill[anchor=north] (20,20.5) --++ (0.5,0) --++ (0,-1) --++ (-1,0) --++ (0,1) --++ (0.5,0);
\draw (20,21) --++ (1,-1) --++ (-1,-1) --++ (-1,1) --++ (1,1);

\fill[white](40,41) --++ (1,-1) --++ (-1,-1) --++ (-1,1) --++ (1,1);
\fill[anchor=north] (40,40.5) --++ (0.5,0) --++ (0,-1) --++ (-1,0) --++ (0,1) --++ (0.5,0);
\draw (40,41) --++ (1,-1) --++ (-1,-1) --++ (-1,1) --++ (1,1);

\fill (20,30.75)  --++ (0.75,-0.75) --++ (-0.75,-0.75) --++ (-0.75,0.75) --++ (0.75,0.75);
\fill (0,10.75)  --++ (0.75,-0.75) --++ (-0.75,-0.75) --++ (-0.75,0.75) --++ (0.75,0.75);

\fill (20,40) circle (0.5);

\fill[white] (20,50.75)  --++ (0.75,0) --++ (0,-1.5) --++ (-1.5,0) --++ (0,1.5) --++ (0.75,0);
\fill (20,50.75) --++ (0.75,-0.75) --++ (-0.75,-0.75) --++ (-0.75,0.75) --++ (0.75,0.75);
\draw (20,50.75)  --++ (0.75,0) --++ (0,-1.5) --++ (-1.5,0) --++ (0,1.5) --++ (0.75,0);

\fill[white](0,30.75)  --++ (0.75,0) --++ (0,-1.5) --++ (-1.5,0) --++ (0,1.5) --++ (0.75,0);
\fill (0,30.75)--++ (0.75,-0.75) --++ (-0.75,-0.75) --++ (-0.75,0.75) --++ (0.75,0.75);
\draw (0,30.75)  --++ (0.75,0) --++ (0,-1.5) --++ (-1.5,0) --++ (0,1.5) --++ (0.75,0);

\draw[anchor=west] (20,50) node {$x_2$};
\draw[anchor=east] (20,40) node {$y_2$};
\draw[anchor=west] (20,30) node {$x_2$};
\draw[anchor=east] (20,20) node {$x_1$};
\draw[anchor=west] (20,10) node {$y_1$};
\draw[anchor=east] (20,0) node {$x_1$};

\draw[anchor=north] (0,10) node {$x_1$};
\draw[anchor=north] (0,30) node {$x_1$};

\draw[anchor=south] (40,20) node {$x_2$};
\draw[anchor=south] (40,40) node {$x_2$};

\fill (5,-5)--++ (0.75,-0.75) --++ (-0.75,-0.75) --++ (-0.75,0.75) --++ (0.75,0.75);
\draw (5,-5)  --++ (0.75,0) --++ (0,-1.5) --++ (-1.5,0) --++ (0,1.5) --++ (0.75,0);
\draw [anchor=north] (5,-7) node {$\alpha_1$};

\fill (15,-5)  --++ (0.75,-0.75) --++ (-0.75,-0.75) --++ (-0.75,0.75) --++ (0.75,0.75);
\draw [anchor=north] (15,-7) node {$\alpha_2$};

\fill[anchor=north] (35,-5) --++ (0.5,0) --++ (0,-1) --++ (-1,0) --++ (0,1) --++ (0.5,0);
\draw (35,-4.5) --++ (1,-1) --++ (-1,-1) --++ (-1,1) --++ (1,1);
\draw [anchor=north] (35,-7) node {$\alpha_4$};

\fill[anchor=north] (25,-5) --++ (0.5,0) --++ (0,-1) --++ (-1,0) --++ (0,1) --++ (0.5,0);
\draw [anchor=north] (25,-7) node {$\alpha_3$};
\end{tikzpicture}
\end{center}
\caption{Defining $ (x_1 <_1 x_2 \Leftrightarrow y_1 < y_2)$}
\end{figure}

\begin{lemma}\label{lemma:1}
If $\aut(M) \models (x_1 <_1 x_2 \Leftrightarrow y_1 < y_2)$ then $M \models x_1 <_M x_2$.
\end{lemma}
\begin{proof}
Let $(x_0,x_1) \in M$ be such that $\aut(M) \models \neg  (x_1 <_0 x_2 \Leftrightarrow y_1 < y_2)$.  We will show that when $\aut(M) \models \alpha_i$ then $x_1 <_M x_2$ for each possible $i$.

First, assume that $\aut(M) \models \alpha_1$.  Since $\aut(M) \models B(y_2;y_1,x_2)$ and we are supposing that $y_1 <_M y_2$, we know that $x_2 >_M y_2$.  We also know that $x_1$ cannot be greater than $x_2$, as otherwise $\aut(M) \models (x_1 <_0 x_2 \Leftrightarrow y_1 < y_2)$.  Since we have asserted that $\aut(M) \models \mathrm{Related}(x_1,x_2)$, this means that $x_1 <_M x_2$.

Now we assume that $\aut(M) \models \alpha_2$, so either $x_1 <_M x_2 <_M y_2$ or $y_2 <_M x_2 <_M x_1$, but the latter contradicts our assertion that not both of $x_1$ and $x_2$ are related to both $y_1$ and $y_2$.

Assume that $\aut(M) \models \alpha_3$, so $x_1 <_M y_1 <_M y_2$.  If $x_2 <_M x_1$ then $x_2 <_M y_1,y_2$, contradicting $\aut(M) \models \neg (x_1 <_0 x_2 \Leftrightarrow y_1 < y_2)$.

Assume that $\aut(M) \models \alpha_2$, so either $x_2 <_M x_1 <_M y_1$ or $y_1 <_M x_1 <_M x_2$, but the former contradicts our assertion that not both of $x_1$ and $x_2$ are related to both $y_1$ and $y_2$.

Assume that $\aut(M) \models \alpha_5$, so $\aut(M) \models \mathrm{Related}(x_1,y_2) \wedge \neg\mathrm{Related}(x_1,y_1)$.  This means that $x_1 <_M y_2$.  If $x_2 <_M x_1$ then $x_2 <_M y_2$, but we have asserted that $\aut(M) \models \neg\mathrm{Related}(x_2,y_2)$.
\end{proof}

\begin{dfn}
Let $n \geq 2$.  The formula $(x_1 <_n x_2 \Leftrightarrow y_1 < y_2)$ is defined to be the conjunction of the following four formulas:
$$
\forall z \left(    \bigwedge\limits_{i<n-1}\neg(x_1 <_i z \Leftrightarrow y_1 < y_2) \right) \wedge \forall z \left(  \bigwedge\limits_{i<n-1}  \neg(z <_i x_2 \Leftrightarrow y_1 < y_2) \right)
$$
to ensure that the order is yet to be resolved for either $x_1$ or $x_2$;
$$
\begin{array}{c}
(\exists z \left(  (x_1 <_{n-1} z \Leftrightarrow y_1 < y_2) \right)) \wedge \forall z \left( \neg  (z <_{n-1} x_2 \Leftrightarrow y_1 < y_2) \right)) \\
\vee \\
(\exists z \left(  (z <_{n-1} x_2 \Leftrightarrow y_1 < y_2) \right)) \wedge \forall z \left( \neg  (x_1 <_{n-1} z \Leftrightarrow y_1 < y_2) \right)) 
\end{array}
$$
to ensure that exactly one of $x_1$ and $x_2$ is related by $<_{n-1}$ to something; 
$$
\begin{array}{c}
\forall z \left( \neg  (z <_{n-1} x_2 \Leftrightarrow y_1 < y_2) \right) \rightarrow \\
\exists w \left( (x_1 <_{n-1} w \Leftrightarrow y_1 < y_2)  \wedge (x_1 <_{1} x_2 \Leftrightarrow x_1 < w) \right)
\end{array}
$$
to describe what happens when $x_1$ is in the area where the order is defined, but $x_2$ is not, and;
$$
\begin{array}{c}
\forall z \left(  \neg  (x_1 <_{n-1} z \Leftrightarrow y_1 < y_2) \right) \rightarrow \\
\exists w \left( (w <_{n-1} x_2 \Leftrightarrow y_1 < y_2)  \wedge (x_1 <_{1} x_2 \Leftrightarrow w < x_2) \right)
\end{array}
$$
to describe what happens when $x_2$ is in the area where the order is defined, but $x_1$ is not.
\end{dfn}

\begin{prop}\label{lemma:<n}
If $\aut(M) \models (x_1 <_n x_2 \Leftrightarrow y_1 < y_2)$ then $M \models x_1 <_M x_2$.
\end{prop}
\begin{proof}
We proceed by induction, starting with the base case $n=2$.   Suppose that
$$\aut(M) \models (x_1 <_2 x_2 \Leftrightarrow y_1 < y_2)$$
If
$$\aut(M) \models \forall z \left( \neg  (z <_{n-1} x_2 \Leftrightarrow y_1 < y_2) \right)$$
then there is a $w \in M$ such that
$\aut(M) \models (x_1 <_{1} w \Leftrightarrow y_1 < y_2)$
and so $x_1 <_M w$.  Therefore $\aut(M) \models (x_1 <_{1} x_2 \Leftrightarrow x_1 < w)$ implies that $x_1 <_{M} x_2$.

If
$\aut(M) \models \forall z \left(  \neg  (x_1 <_{1} z \Leftrightarrow y_1 < y_2) \right)$
then there is a $w \in M$ such that
$$\aut(M) \models (w <_{1} x_2 \Leftrightarrow y_1 < y_2)$$
and so $w <_{M} x_2 $.  Therefore $\aut(M) \models (x_1 <_{1} x_2 \Leftrightarrow w < x_2)$ implies that $x_1 <_{M} x_2$.

We now examine the induction step.  Suppose that if
$\aut(M) \models (x_1 <_{n-1} x_2 \Leftrightarrow y_1 < y_2)$
then $x_1 <_M x_2$ and also suppose that
$\aut(M) \models (x_1 <_{n} x_2 \Leftrightarrow y_1 < y_2)$.

If $\aut(M) \models \forall z \left( \neg  (z <_{n-1} x_2 \Leftrightarrow y_1 < y_2) \right)$
then there is a $w \in M$ such that
$$\aut(M) \models (x_1 <_{n-1} w \Leftrightarrow y_1 < y_2)$$
and so $x_1 <_M w$.  Therefore
$\aut(M) \models (x_1 <_{1} x_2 \Leftrightarrow x_1 < w)$
implies that $x_1 <_{M} x_2$.

If
$\aut(M) \models \forall z \left( \neg  (x_1 <_{n-1} z \Leftrightarrow y_1 < y_2) \right)$
then there is a $w \in M$ such that
$$\aut(M) \models (w <_{n-1} x_2 \Leftrightarrow y_1 < y_2)$$ and so $w <_M x_2$.  Therefore
$\aut(M) \models (x_1 <_{1} x_2 \Leftrightarrow w < x_2)$ implies that $x_1 <_{M} x_2$.
\end{proof}

\begin{dfn}
$(x_1 < x_2 \Leftrightarrow y_1 < y_2)$ is defined to be the $L_{\omega_1,\omega}$-formula:
$$\bigvee\limits_{n < \omega} (x_1 <_n x_2 \Leftrightarrow y_1 < y_2)$$
\end{dfn}

\begin{theorem}
$\aut(M) \models (x_1 < x_2 \Leftrightarrow y_1 < y_2)$ if and only if $M \models x_1 <_M x_2$.
\end{theorem}
\begin{proof}
Suppose $\aut(M) \models (x_1 < x_2 \Leftrightarrow y_1 < y_2)$.  In the first clause of their definition, we ensured that each of the formulas $(x_1 <_n x_2 \Leftrightarrow y_1 < y_2)$ are mutually exclusive.  By Lemmas \ref{lemma:0} and \ref{lemma:1} and Proposition \ref{lemma:<n}, so no matter which $\aut(M)$ realises, we have ensured that $x_1 <_M x_2$.

Suppose $M \models x_1 <_M x_2$.  We examine the length of $\path{\lbrace x_1,x_2 \rbrace , \lbrace y_1,y_2 \rbrace }$, which we shall call $j$.  If $j \leq 2$ then at least one of $x_1$ and $x_2$ is related to at least one of $y_1$ and $y_2$.

Suppose that both $x_1$ and $x_2$ are related to both $y_1$ and $y_2$.  Since $x_1$ and $x_2$ must occur in one of the situations described by $(x_1 <_0 x_2 \Leftrightarrow y_1 < y_2)$.

Now suppose that not both of $x_1$ and $x_2$ are related to both $y_1$ and $y_2$, but at least one is.  This situation is fully described by $(x_1 <_1 x_2 \Leftrightarrow y_1 < y_2)$.

If neither $x_1$ nor $x_2$ are related to either $y_1$ or $y_2$ then $\aut(M)$ realises neither
$$
\begin{array}{c c c}
(x_1 <_0 x_2 \Leftrightarrow y_1 < y_2) & \mathrm{nor} & (x_1 <_1 x_2 \Leftrightarrow y_1 < y_2)
\end{array}
$$
as both of those formulas contain instances of $B(z;x,y)$ that prevent this.

Now suppose that $j \geq 3$.  We also assume that for all $z_1$ and $z_2$ such that $z_1 <_M z_2$, the length of $\path{ \lbrace z_1,z_2 \rbrace , \lbrace y_1, y_2 \rbrace }$ is $i$ for $i<j$ if and only if $$\aut(M) \models (z_1 <_i z_2 \Leftrightarrow y_1 < y_2)$$

Suppose $\path{ \lbrace x_1,x_2 \rbrace , \lbrace y_1, y_2 \rbrace }=j$.  We first wish to show that $(x_0,x_1)$ satisfies the first clause of $(x_1 < x_2 \Leftrightarrow y_1 < y_2)$, vis.
$$
\forall z \left(  \bigwedge\limits_{i<j-1}\neg(x_1 <_i z \Leftrightarrow y_1 < y_2) \right) \wedge \forall z \left(  \bigwedge\limits_{i<n-1}  \neg(z <_i x_2 \Leftrightarrow y_1 < y_2) \right)
$$
If there is a $z$ such that $(x_1 <_i z \Leftrightarrow y_1 < y_2) )$ for some $i < j-1$ then, by the induction hypothesis, the length of $\path{ \lbrace z,x_1 \rbrace , \lbrace y_1, y_2 \rbrace }$ is less than $j-1$.  Since $x_1 <_M x_2$, this means that the length of $\path{ \lbrace x_1,x_2 \rbrace , \lbrace y_1, y_2 \rbrace }$ is less than $j$, contradicting our assumptions.  If there is a $z$ such that $(z <_i x_2 \Leftrightarrow y_1 < y_2) )$ then we reach a similar contradiction.

Let us now examine the second clause:
$$
\begin{array}{c}
\forall z \left( \neg  (z <_{j-1} x_2 \Leftrightarrow y_1 < y_2) \right) \rightarrow \\
\exists w \left( (x_1 <_{j-1} w \Leftrightarrow y_1 < y_2)  \wedge (x_1 <_{1} x_2 \Leftrightarrow x_1 < w) \right)
\end{array}
$$
If there is a $z$ such that $\aut(M) \models (z <_{j-1} x_2 \Leftrightarrow y_1 < y_2)$ then we are done, so suppose that there is no such $z$.  Let $z_1, \ldots, z_j$ be the elements of the connecting set of $\path{ \lbrace x_1,x_2 \rbrace, \lbrace y_1,y_2 \rbrace}$, such that $z_1$ is related to $x_1$ and $x_2$.  

If $z_2 <_M z_1 \leq_M x_2$ then $\aut(M) \models (z_2 <_{j-1} x_2 \Leftrightarrow y_1 < y_2)$, so we may assume that $x_1 \leq_M z_1 <_M z_2$.

$\path{\lbrace z_1,z_2 \rbrace, \lbrace y_1,y_2 \rbrace} = \path{ z_2 , \lbrace y_1,y_2 \rbrace}$ has length $j-1$, so $\aut(M) \models (x_1 <_{j-1} z_2 \Leftrightarrow y_1 < y_2)$.  Additionally, we have deduced that $z_2 \parallel x_2$, and $x_1 \leq z_2,x_2$, so $\aut(M) \models (x_1 <_{1} x_2 \Leftrightarrow x_1 < z_2)$, so $(x_1,x_2)$ satisfies the second clause.

Recall that the third clause we must examine is:
$$
\begin{array}{c}
\forall z \left(  \neg  (x_1 <_{n-1} z \Leftrightarrow y_1 < y_2) \right) \rightarrow \\
\exists w \left( (w <_{n-1} x_2 \Leftrightarrow y_1 < y_2)  \wedge (x_1 <_{1} x_2 \Leftrightarrow w < x_2) \right)
\end{array}
$$
If there is a $z$ such that $\aut(M) \models (x_1 <_{j-1} z \Leftrightarrow y_1 < y_2)$ then we are done, so suppose that there is no such $z$.  Again, let $z_1, \ldots, z_j$ be the elements of the connecting set of $\path{ \lbrace x_1,x_2 \rbrace, \lbrace y_1,y_2 \rbrace}$, such that $z_1$ is related to $x_1$ and $x_2$.  

If $x_1 \leq_M z_1 <_M z_2$ then $\aut(M) \models (x_1 <_{j-1} z_2 \Leftrightarrow y_1 < y_2)$, so we may assume that $z_2 <_M z_1 \leq_M x_2$.

$\path{\lbrace z_1,z_2 \rbrace, \lbrace y_1,y_2 \rbrace} = \path{ z_2 , \lbrace y_1,y_2 \rbrace}$ has length $j-1$, so
$$\aut(M) \models (z_2 <_{j-1} x_2 \Leftrightarrow y_1 < y_2)$$
Additionally, we have deduced that $z_2 \parallel x_1$, and $x_2 \geq z_2,x_1$, so $\aut(M) \models (x_1 <_{1} x_2 \Leftrightarrow z_2 < x_1)$, so $(x_1,x_2)$ satisfies the third clause.

Now suppose that $\aut(M) \models (x_1 <_{j} x_2 \Leftrightarrow y_1 < y_2)$.  Let $z_k, \ldots, z_j$ be the elements of the connecting set of $\path{ \lbrace x_1,x_2 \rbrace, \lbrace y_1,y_2 \rbrace}$, such that $z_k$ is related to $x_1$ and $x_2$.  Since
$$
\aut(M) \models \forall z \left(  \bigwedge\limits_{i<j-1}\neg(x_1 <_i z \Leftrightarrow y_1 < y_2) \right) \wedge \forall z \left(  \bigwedge\limits_{i<n-1}  \neg(z <_i x_2 \Leftrightarrow y_1 < y_2) \right)
$$
the length of $\path{ \lbrace x_1,x_2 \rbrace, \lbrace y_1,y_2 \rbrace}$ has at least $j$ elements, and thus $k \leq 1$.  By the induction hypothesis, either
$$
\begin{array}{c c c}
 \aut(M) \models (z_1 <_{j-1} z_2 \Leftrightarrow y_1 < y_2) & \mathrm{or} & \aut(M) \models (z_2 <_{j-1} z_1 \Leftrightarrow y_1 < y_2) 
\end{array}
$$
so if $k \not= 1$ then $(x_1,x_2)$ cannot possibly satisfy the second and third coordinates.
\end{proof} 

\bibliography{bib}{}
\bibliographystyle{plain}

\end{document}